\documentclass[11pt]{article}
\usepackage{latexsym,amssymb,amsmath,amsfonts,amsthm}
\usepackage{graphics,graphicx,subfigure}
\usepackage{mathrsfs,cases,cite}
\newtheorem{theorem}{Theorem}[section]
\newtheorem{lemma}[theorem]{Lemma}
\newtheorem{corollary}[theorem]{Corollary}

\newtheorem{remark}[theorem]{Remark}

\newcommand{\mat}{\mathbb}
\newcommand{\R}{{\mat R}}

\newcommand{\C}{{\mat C}}

\newcommand{\ds}{\displaystyle}
\newcommand{\no}{\nonumber}
\newcommand{\Om}{\Omega}
\newcommand{\om}{\omega}
\newcommand{\ba}{\backslash}
\newcommand{\pa}{\partial}
\newcommand{\al}{\alpha}
\newcommand{\g}{\gamma}

\newcommand{\la}{\lambda}
\newcommand{\La}{\Lambda}

\newcommand{\ov}{\overline}
\newcommand{\na}{\nabla}
\newcommand{\de}{\delta}
\newcommand{\De}{\Delta}
\newcommand{\sig}{\sigma}
\newcommand{\wid}{\widetilde}
\newcommand{\bs}{\boldsymbol}
\newcommand{\PML}{{\rm PML}}
\providecommand{\Div}{\operatorname{div}}
\providecommand{\Rp}{\operatorname{Re}}
\providecommand{\Ip}{\operatorname{Im}}
\providecommand{\Trace}{\operatorname{tr}}
\providecommand{\Diag}{\operatorname{diag}}

\begin{document}
\renewcommand{\theequation}{\arabic{section}.\arabic{equation}}
	
\title{\bf Convergence of the PML method for thermoelastic wave scattering problems}
\author{Qianyuan Yin\thanks{Academy of Mathematics and Systems Science, Chinese Academy of Sciences,
Beijing 100190, China and School of Mathematical Sciences, University of Chinese Academy of Sciences,
Beijing 100049, China (yinqianyuan@amss.ac.cn; co-first author).}
\and
Changkun Wei\thanks{School of Mathematics and Statistics, Beijing Jiaotong University,
			Beijing 100044, China (ckwei@bjtu.edu.cn; co-first author, corresponding author).}
\and
Bo Zhang\thanks{SKLMS and Academy of Mathematics and Systems Science, Chinese Academy of Sciences,
Beijing 100190, China and School of Mathematical Sciences, University of Chinese Academy of Sciences,
Beijing 100049, China (b.zhang@amt.ac.cn).}
}
\date{}
	
\maketitle

\begin{abstract}
This paper is concerned with the thermoelastic obstacle scattering problem in three dimensions.
A uniaxial perfectly matched layer (PML) method is firstly introduced to truncate the unbounded scattering
problem, leading to a truncated PML problem in a bounded domain. Under certain constraints on model parameters,
the well-posedness for the truncated PML problem is then proved except possibly for a discrete set of frequencies,
based on the analytic Fredholm theory. Moreover, the exponential convergence of the uniaxial PML method is 
established in terms of the thickness and absorbing parameters of PML layer. The proof is based on the PML 
extension technique and the exponential decay properties of the modified fundamental solution.
As far as we know, this is the first convergence result of the PML method for the time-harmonic
thermoelastic scattering problem.

\vspace{.2in}

\end{abstract}

\section{Introduction}
\setcounter{equation}{0}

In this paper, we study the thermoelastic obstacle scattering problem with the Dirichlet boundary condition.
In the pure elastic theory of the geophysical applications, the effect of temperature on the medium is often 
ignored and only the mechanical properties of elastic solids is considered. However, to more accurately model 
real-world scenarios such as subsurface exploration, the thermal effects must be taken into account when 
interpreting geophysical data. For example, the hot dry rock geothermal model can be used to obtain the thermal 
information of media by the thermoelastic wave field. The thermoelastic model reveals a connection between the 
thermal and mechanical motions. We refer to \cite{ASZ2024,BXY2019,Cakoni2000,Kupradze1979,wang2025} for the 
mathematical analysis and numerical simulations for the thermoelastic wave scattering problems.

In practical numerical solutions to wave scattering problems, one usually needs to truncate the unbounded domain
in which the wave propagates into a bounded computational domain. The perfectly matched layer (PML) is an efficient
and effective numerical truncation technique, which was originally proposed by B\'erenger in 1994 for solving the
time-dependent Maxwell's equations \cite{Berenger1994}. The basic idea of the PML method is to construct an artificial
layer surrounding the computational domain to effectively attenuate outgoing waves. In general, the designed medium 
needs to make sure the scattered waves decay rapidly (exponentially) in the absorbing layer of finite thickness 
regardless of the wave incident angles. The PML method has an excellent numerical performance and has been widely used
in various types of wave scattering problems; see \cite{Fathi2015,Lei2023,Pakravan2014,Pakravan2017} for the numerical
solution of forward models in the full waveform inversion problems in geophysics and \cite{Hou2021,Wang2022,Yang2024}
for the numerical simulations in thermoelastic and poroelastic wave scattering problems.

Extensive research works are available on the convergence analysis of PML methods. The convergence of the circular and
spherical PML methods is studied in \cite{BW2005,Bramble2007,Chen2009,CL2005,CZ2017,HSZ2003,Lassas1998,wyz2020}, and
convergence of the uniaxial (Cartesian) PML methods can be found in \cite{BP2012,BP2013,CCZ2013,CW2008,CW2012,CZ2010,
Kim2010,wyz2021} for acoustic and electromagnetic wave scattering problems. For the elastic wave scattering problems, 
the exponential convergence of a spherical PML method was established in \cite{BPT2010}, and the exponential convergence
of a Cartesian PML method was established in \cite{CXZ2016}. In the remarkable work \cite{CXZ2016}, the authors consider
the Cartesian PML problem with the mixed boundary condition at the outer boundary of the truncated domain to further 
extend the reflection argument in the \cite{BP2012,BP2013}. This work provides several important analysis techniques 
and innovation, such as the integral representation formula with PML and the limiting absorption principle for elastic 
wave scattering solution. For the elastic scattering problems in periodic structures, the exponential convergence for 
the proposed PML method was established in \cite{Jiang2017,Jiang2018,Zhu2024} for both one-dimensional and two-dimensional
 cases. To the best of our knowledge, there is no convergence result on the PML method for the thermoelastic wave 
 scattering problems, especially for three-dimensional case.

The purpose of this paper is to study the convergence of the uniaxial PML method for the three-dimensional 
time-harmonic thermoelastic wave scattering problems. We adopt the complex coordinate stretching in the 
following form \cite{CXZ2016,Chew1994}
\begin{equation*}
	\wid{x}_j(x_j)=x_j+z\int_0^{x_j} \al_j(s)ds,\;z=\zeta+{\bf i},\;j=1,2,3,
\end{equation*}
where $\zeta\geq 0$ is a constant to be specified, ${\bf i}$ denotes the imaginary unit and $\al_j(s)$ is the PML medium
property. Motivated by the work
\cite{CXZ2016}, the parameter $\zeta$ can be appropriately chosen to ensure the ellipticity of the elastic part of the
thermoelastic PML operator. We remark that the PML analysis for the thermoelastic wave scattering problem is more 
challenging compared with the purely elastic problem, due to the coexistence of fast and slow compressional waves and
shear waves, with both complex and real wavenumbers. More restrictions on the PML parameters need to be imposed to 
obtain the well-posedness and exponential convergence for the PML method. Besides, the method for the well-posedness 
analysis of elastic PML problems in \cite{CXZ2016} cannot be directly applied to the thermoelastic problem, since the
symmetry in the thermoelastic PML variational formulation is lacking. Instead, we employ analytic Fredholm theory and
carefully handle parameter dependencies, together with appropriate choices of frequency to prove the well-posedness 
of the PML problem under certain conditions. The exponential convergence of the PML method is finally established in
terms of the thickness and parameters of the PML layer, based on the PML extension and the exponential decay properties
of the modified fundamental solution.

The remaining part of the paper is organized as follows. In section 2, we present the mathematical model for the 
time-harmonic thermoelastic scattering problem and recall the results on the existence and uniqueness of solutions. 
Section 3 is devoted to propose the PML method for the thermoelastic wave equation and estimate the thermoelastic 
fundamental solution matrix elements with PML extension. In section 4, the well-posedness for the thermoelastic PML 
equation in the truncated domain and layer are carefully studied and analyzed. In section 5, we give a hard analysis 
of the $H^{1/2}$ norm of the PML extension potential on the outer boundary of the layer and further show the error estimate
between the Dirichlet-to-Neumann (DtN) operators of the original scattering problem and the truncated PML problem. 
The exponential convergence of the PML method is finally established in terms of the thickness and parameters of the PML layer.

\section{Thermoelastic scattering problem}\label{section_model}
\setcounter{equation}{0}

In this section, we give a brief mathematical description of the thermoelastic scattering model and recall the 
results on the existence and uniqueness of solutions.

We begin with an outline of the geometric configuration of the model problem. Let $\Om\subseteq\R^3$ be a bounded domain 
with smooth boundary $\pa\Om$. Assume that $\R^3\ba\ov{\Om}$ is filled with a homogeneous and an isotropic elastic medium,
characterized by the constant mass density $\rho$ and Lam\'{e} parameters $\la$, $\mu$ satisfying that $\mu>0$ and 
$3\la+2\mu>0$, and the thermal coupling constants $\g$ and $\eta$ given by
\begin{equation*}
	\g = (3\la+2\mu)\al_T,\quad \eta=\frac{T_0\g}{\la_0},
\end{equation*}
where $\al_T$ is the coefficient of volumetric thermal expansion, $T_0$ represents a reference temperature, and $\la_0$ 
denotes the thermal conductivity coefficient.

The thermoelastic scattering problem with homogeneous Dirichlet boundary condition can be then modeled by the following 
time-harmonic Biot's system of linearized thermoelasticity \cite{Kupradze1979}
\begin{subequations}
	\begin{align}\label{thermoelastic1}
		\De^*\bs{u}+\rho \om^2 \bs{u}-\g \na p={\bf Q} \quad &{\rm in}\;\;\R^3\ba\ov{\Om},\\\label{thermoelastic2}
		\De p+qp+{\bf{i}}\om \eta \na \cdot \bs{u}=0 \quad &{\rm in}\;\;\R^3\ba\ov{\Om},\\\label{DBC}
		(\bs{u}^{\top},p)^{\top}=0 \quad &{\rm on}\;\; \pa \Om,
	\end{align}
\end{subequations}
where $\bs{u}=(u_1,u_2,u_3)^{\top}$ is the elastic displacement vector, $p$ is the temperature field, 
$q={\bf{i}}\om/\kappa$ with $\kappa> 0$ representing the thermal diffusivity, and the Lam\'{e} operator $\De^*$ defined by
\begin{equation*}
	\De^*\bs{u}:=\mu \De \bs{u}+(\la +\mu)\na \na \cdot \bs{u} =\na \cdot \sig(\bs{u}).
\end{equation*}
Here, the symmetric stress tensor $\sig(\bs{u})$ is given by the generalized Hooke's law
\begin{equation*}
	\sig(\bs{u})=(\la \na \cdot \bs{u} )\mat{I}+2\mu \epsilon(\bs{u}), \quad \epsilon(\bs{u})=\frac{1}{2}
	(\na \bs{u}+(\na \bs{u})^{\top}),
\end{equation*}
where $\mat{I}$ is the identity matrix and $\na\bs{u}$ denotes the displacement gradient tensor:
\begin{equation*}
	\na\bs{u}=\left[
	\begin{matrix}
		 \pa_{x_1}u_1 & \pa_{x_2}u_1 & \pa_{x_3} u_1 \\
		 \pa_{x_1}u_2 & \pa_{x_2}u_2 & \pa_{x_3} u_2 \\
		 \pa_{x_1}u_3 & \pa_{x_2}u_3 & \pa_{x_3} u_3
	\end{matrix}\right].
\end{equation*}
Throughout the paper, we always assume that the source ${\bf Q}$ has compact support inside 
$B_1=\{(x_1,x_2,x_3)^{\top} \in \R^3:\lvert x_j \rvert <l_j,j=1,2,3\}$ with some constants 
$l_j>0,j=1,2,3$. The matrix form of the problem \eqref{thermoelastic1}-\eqref{DBC} can be 
alternatively obtained
\begin{equation}\label{matrix_eq}
	\begin{aligned}
		LU=({\bf Q}^{\top},0)^{\top}, \quad &{\rm in}\;\;\R^3\ba\ov{\Om},\\
		U=(\bs{u}^{\top},p)^{\top}=0, \quad &{\rm on}\;\; \pa \Om,
	\end{aligned}
\end{equation}
where
\begin{equation*}
	L=\begin{bmatrix}
	(\De^* +\rho \om^2)\mat{I} & -\g \na \\
	q \eta \kappa \na \cdot &  \De +q
\end{bmatrix}.
\end{equation*}
For a smooth solution $U$ of the homogeneous thermoelastic system in $\R^3\ba\ov{B}_1$, using 
Helmholtz decomposition we have $\bs{u}=\na \phi +\bs{\psi}$ with $\Div\bs{\psi}=0$. 
By substituting it into homogeneous thermoelastic system \eqref{matrix_eq}, we can obtain equations
\begin{subequations}
\begin{align}
	(\la+2\mu)\De \phi +\rho \om^2 \phi -\g p&=0,\label{HeD1}\\
	-\mu \na \times \na \times \bs{\psi} +\rho \om^2 \bs{\psi} &=0, \label{HeD2}\\
	\De p + qp +{\bf{i}}\om \eta \De \phi &=0.\label{HeD3}
\end{align}
\end{subequations}
Substituting the equation \eqref{HeD1} into the equation \eqref{HeD3} yields
\begin{equation}\label{HeD4}
	\De^2 \phi +\left(\frac{\rho\om^2}{\la+2\mu }+\frac{{\bf{i}}\om\eta \g}{\la+2\mu }+q\right) \De \phi
	+\frac{qp\om^2}{\la+2\mu }\phi=0
\end{equation}
Rewrite the equation \eqref{HeD4} in the following form
\begin{equation}\label{HeD5}
	(\De+\la_1^2) (\De+\la_2^2)\phi=0,
\end{equation}
where
\begin{equation}\label{chara}
	\la_1^2+\la_2^2=\frac{\rho\om^2}{\la+2\mu }+\frac{{\bf{i}}\om\eta \g}{\la+2\mu }+q,
	\quad\la_1^2\la_2^2=\frac{qp\om^2}{\la+2\mu }.
\end{equation}
Noting that $k_p=\ds\sqrt{\frac{\rho\om}{\la +2\mu}}$ is the wave number of compressional waves, 
we reformulate equations \eqref{chara} into the following characteristic relation equations
\begin{equation}\label{chara1}
	 \la_1^2+\la_2^2=\frac{{\bf{i}}\om}{\kappa}+\frac{{\bf{i}}\om\g \eta}{\la+2\mu}+k_p^2,
	 \quad\la_1^2\la_2^2=\frac{{\bf{i}}\om}{\kappa}k_p^2.
\end{equation}
The thermoelastic wave field obeys the following decomposition lemma \cite{BXY2019}.
\begin{lemma}\label{lem_decom}
The solution $U$ of homogeneous thermoelastic system in $\R^3\ba\ov{B}_1$
can be written in the form
\begin{equation*}
	\bs{u}=\bs{u}^1+\bs{u}^2+\bs{u}^3,\quad p=p^1+p^2,
\end{equation*}
where $\bs{u}^1,\bs{u}^2,\bs{u}^3$ and $p^1,p^2$ satisfy
\begin{equation*}
	\begin{aligned}
		(\De +\la_1^2)\bs{u}^1&=0 ,\quad & (\De +\la_2^2)\bs{u}^2&=0,\quad& (\De +\la_3^2)\bs{u}^3&=0, \\
		\na \times \bs{u}^1&=0,\quad & \quad \na \times \bs{u}^2&=0,\quad & \na\cdot \bs{u}^3&=0, \\
		(\De +\la_1^2)p^1&=0 ,\quad & (\De +\la_2^2)p^2&=0.
	\end{aligned}
\end{equation*}
Here, the wave number $\la_1,\la_2$ are the roots of the characteristic relation equations 
\eqref{chara1}, and $\la_3=\ds\sqrt{\frac{\rho}{\mu}}\om$ is the the wave number of shear waves.
\end{lemma}	
In addition, the scattered wave field $U$ is assumed to satisfy the Kupradze radiation conditions 
as $r=\sqrt{x_1^2+x_2^2+x_3^2}\rightarrow \infty$ for $i=1,2,3$ and $j=1,2$ (cf. \cite{Kupradze1979}):
\begin{equation}\label{Kupradze}
\begin{aligned}
	\bs{u}^j=o(r^{-1}) ,\quad & \pa_{x_i}\bs{u}^j=O(r^{-2}), \\
	p^j=o(r^{-1}) ,\quad & \pa_{x_i}p^j=O(r^{-2}), \\
	\bs{u}^3=o(r^{-1}) ,\quad & r(\pa_{r}\bs{u}^3-i\la_3\bs{u}^3)=O(r^{-1}).
\end{aligned}
\end{equation}
For given $\bs{f} \in H^{\frac{1}{2}}(\pa B_1)^4$, the exterior Dirichlet problem
\begin{equation}\label{eDp1}
	\begin{cases}
		\De^*\bs{u}+\rho \om^2 \bs{u}-\g \na p=0 \quad {\rm in}\;\;\R^3\ba\ov{B}_1,\\
		\De p+qp+{\bf{i}}\om \eta \na \cdot \bs{u}=0 \quad {\rm in}\;\;\R^3\ba \ov{B}_1,\\
		(\bs{u}^{\top},p)^{\top}=\bs{f} \quad {\rm on} \;\;\pa B_1,\\
		(\bs{u}^{\top},p)^{\top}\;\text{satisfies the Kupradze radiation conditions \eqref{Kupradze} at infinity},
	\end{cases}
\end{equation}
admits a unique weak solution $U=(\bs{u}^{\top},p)^{\top}$ in the space $H_{loc}^1(\R^3\ba\ov{\Om})^4$ 
(cf. \cite{Cakoni2000,Kupradze1979}). Hence, we can define the Dirichlet to Neumann (DtN) map 
$\mathcal{N}:H^{1/2}(\pa B_1)^4\to H^{-1/2}(\pa B_1)^4$ for problem \eqref{eDp1} by 
\begin{equation}\label{dtn}
	\mathcal{N}\bs{f}=\mathcal{R}(\bs{u}^{\top},p^{\top})|_{\pa B_1}, {\rm where}\;\mathcal{R}(\pa,\nu)U=\begin{bmatrix}
		T(\pa,\nu) & -\g \nu \\
		0  &    \pa_{\nu}	
	\end{bmatrix}U\quad{\rm on}\;\pa B_1.
\end{equation}
Here, the traction operator $T(\pa,\nu)\bs{u}=\sig(\bs{u})\nu$, and $\nu=(\nu_1,\nu_2,\nu_3)^{\top}$
denotes the unit outward normal vector of the boundary $\pa B_1$. The well-posedness of \eqref{eDp1} 
implies that $\mathcal{N}$ is a continuous linear operator.

According to \cite{Kupradze1979}, the $4\times 4$ matrix-valued fundamental solution
$\Phi(x,\om)=(\Phi_{ij}(x,\om))$ is defined by
\begin{equation*}
	\begin{aligned}
		\Phi_{ij}(x,\om)=&\sum_{l=1}^3\Bigg[(1-\de_{i4})(1-\de_{j4})\Big(\frac{\de_{ij}}{2\pi \mu}\de_{3l}
		-\al_{l}\frac{\pa^2}{\pa x_{i}\pa{x_j}}\Big)\\
		&+i\beta_l\om \eta \de_{i4}(1-\de_{j4})\frac{\pa}{\pa{x_j}}
		-\beta_l\g \de_{j4}(1-\de_{i4})\frac{\pa}{\pa{x_i}}
		+\de_{i4}\de_{j4}\g_{l}\Bigg]\frac{\exp({\bf i}\la_{l}\lvert x \rvert)}{\lvert x \rvert},
	\end{aligned}
\end{equation*}
with the parameters
\begin{equation*}
	\begin{cases}
		\ds \al_l =\frac{(-1)^l(1-{\bf{i}}\om\kappa^{-1}\la_l^{-2})(\de_{1l}+\de_{2l})}{2\pi(\la+2\mu)
		(\la_2^2-\la_1^2)}-\frac{\de_{3l}}{2\pi \rho \om^2},\\
		\ds \beta_l =\frac{(-1)^l(\la_l^2-k_p^2)(\de_{1l}+\de_{2l})}{2\pi(\la_2^2-\la_1^2)},\\
		\ds \g_l =\frac{(-1)^l(\de_{1l}+\de_{2l})}{2\pi(\la+2\mu)(\la_2^2-\la_1^2)},
	\end{cases}
\end{equation*}
where $\de_{ab}$ denotes the Kronecker symbol of two integer variables $a$ and $b$, and $\la_i,i=1,2,3$
are the wave numbers appeared in Lemma \ref{lem_decom}.

The entries $\Phi_{ij}(x,\om)$ can be written as the simplified form
\begin{equation*}
	\Phi_{ij}(x,\om)=\sum_{l=1}^3\left[C_{1l}f_{\la_l}(\lvert x \rvert)+\frac{C_{2l}x_i+C_{3l}x_j}
	{\lvert x \rvert}f^{'}_{\la_l}(\lvert x \rvert)+C_{4l}\left(\frac{x_ix_j}{\lvert x \rvert^2}
	f^{''}_{\la_l}(\lvert x \rvert)-\frac{x_ix_j}{\lvert x \rvert^3}f^{'}_{\la_l}(\lvert x \rvert)\right)\right],
\end{equation*}
where $f_{\la_l}(r)=\ds e^{{\bf i}\la_{l}r}/r,l=1,2,3$ and $C_{tl},t=1,2,3,4$ are constants.

In what follows, we introduce the thermoelastic single-layer potential for $\bs{p}\in H^{-1/2}(\pa B_1)^4$
\begin{equation*}
	\Psi_{SL}(\bs{p})(x)=\int_{\pa B_1}\Phi(x-y,\om)\bs{p}(y)ds(y),\quad x\in\R^3\ba\ov{B}_1
\end{equation*}
and the double-layer potential for $\bs{q}\in H^{1/2}(\pa B_1)^4$
\begin{equation*}
	\Psi_{DL}(\bs{q})(x)=\int_{\pa B_1}[\mathcal{S}(\pa_y,\nu(y),\om)\Phi^{\top}(x-y,\om)]^{\top}
	\bs{q}(y)ds(y),\quad x\in\R^3\ba\ov{B}_1
\end{equation*}
where the $4\times 4$ matrix differential operator $\mathcal{S}(\pa_y,\nu(y),\om)
=(\mathcal{S}_{ij}(\pa_y,\nu(y),\om))$
is defined by
\begin{equation*}
	\begin{aligned}
		\mathcal{S}_{ij}(\pa_y,\nu(y),\om)=\;&(1-\de_{i4})(1-\de_{j4})\Big(\de_{ij}\mu
		\frac{\pa}{\pa \nu(y)}+\la \nu_i(y)\frac{\pa}{\pa x_j}+\mu \nu_j(y)\frac{\pa}{\pa x_i}\Big)\\
		&+{\bf{i}}\om \eta \de_{j4}(1-\de_{i4})\nu_i(x) + \de_{i4}\de_{j4}\frac{\pa}{\pa{\nu(y)}}.
	\end{aligned}
\end{equation*}
The unique solution of problem \eqref{eDp1} has the integral representation \cite{Kupradze1979}
\begin{equation}
	E(\bs{f})=\frac{1}{2}\Big[\Psi_{SL}(\mathcal{N}\bs{f})-\Psi_{DL}(\bs{f})\Big]\quad {\rm in}\;\R^3\ba\ov{B}_1.
\end{equation}
The well-posedness of the problem \eqref{thermoelastic1}-\eqref{DBC} with Kupradze radiation conditions
\eqref{Kupradze} can be concluded as the following theorem (cf. \cite{Cakoni2000,Kupradze1979}).
\begin{theorem}\label{directscattering}
For given ${\bf Q} \in H^1(\R^3\ba\ov{\Om})^3$ with compact support inside $B_1$, the thermoelastic scattering
problem \eqref{matrix_eq} with Kupradze radiation conditions \eqref{Kupradze} has at most one weak solution
$U=(\bs{u}^{\top},p)^{\top}$ in the space $H_{loc}^1(\R^3\ba\ov{\Om})^4$.
\end{theorem}
Introduce the space $H^{1}_{\pa B_1}(\Om_1)^4=\{\bs{v}\in H^1(\Om_1)^4:\bs{v}=0 \;{\rm on} \;\pa \Om\}$ with 
the bounded domain $\Om_1=B_1\ba\ov{\Om}$. We define the sesquilinear form $\mathcal{B}_1:H^{1}_{\pa B_1}(\Om_1)^4
\times H^{1}_{\pa B_1}(\Om_1)^4\to\C $
\begin{equation*}
	\begin{aligned}
		\mathcal{B}_1(\Phi,\Psi)=&\int_{\Om_1}\Big[\sig(\bs{u})\colon \na\ov{\bs{u}}'-\rho \om^2\bs{u}\cdot
		\ov {\bs{u}}'-\g p\na\cdot \ov{\bs{u}}' +\na \ov{p}'\cdot\na p-qp\ov{p}'\\
		&\qquad-{\bf{i}}\om \eta \ov{p}'\na \cdot \bs{u}\Big]dx-\langle\mathcal{N}\Phi,\Psi\rangle_{\pa B_1},
	\end{aligned}
\end{equation*}
where $\Phi=(\bs{u}^{\top},p)^{\top}$, $\Psi=(\bs{u}'^{\top},p')^{\top}$ and 
$\langle\mathcal{N}\Phi,\Psi\rangle_{\pa B_1}=\ds\int_{\pa B_1}\mathcal{N}\Phi\cdot \ov{\Psi} ds$
denotes the dual pair between the spaces $H^{-1/2}(\pa B_1)^4$ and $H^{1/2}(\pa B_1)^4$.

We rewrite the scattering problem \eqref{thermoelastic1}-\eqref{DBC} into the following variational form:
\begin{equation}
	\mathcal{B}_{1}(\Phi,\Psi)=-\int_{\Om_1}{\bf Q}\cdot \ov{\Psi}dx,\quad \forall\;\Psi \in H^{1}_{\pa B_1}(\Om_1)^4.
\end{equation}
Combining the well-posedness of the scattering problem \eqref{thermoelastic1}-\eqref{DBC} and the open mapping 
theorem, we know that there exists a constant $C>0$ such that the inf-sup condition holds:
\begin{equation}\label{infsup_B1}
	\sup_{\Psi\in {H^{1}_{\pa B_1}(\Om_1)^4}\ba\{0\}} \frac{\mathcal{B}_1(\Phi,\Psi)}
	{\Vert \Psi\Vert_{H^{1}(\Om_1)^4}}\geq C\Vert \Phi\Vert_{H^{1}(\Om_1)^4},\quad \forall\;\Phi\in H^{1}(\Om_1)^4.
\end{equation}
\begin{remark}[Constant Convention]
{\rm Here and in the sequel, the symbol $C$ denotes a generic positive constant which is allowed to 
change from one occurrence to the next.}
\end{remark}

\section{The PML method for thermoelastic wave }
\setcounter{equation}{0}
Let $B_2=\{(x_1,x_2,x_3)^{\top} \in \R^3:\lvert x_j \rvert <l_j+d_j,j=1,2,3\}$ be a cuboid domain
surrounding $B_1$. The truncated PML domain and the PML layer are denoted by $\Om_2:=B_2 \ba \ov{\Om}$
and $\Om_{\PML}:=B_2 \ba\ov{B_1}$, respectively.

We introduce even functions $\al_j(t)\in C^2(\R),j=1,2,3$ satisfying $\al_j'(t)\geq 0$ for $t \geq 0$, 
$\al_j(t)=0$ for $\lvert t \rvert \leq l_j$, and $\al_j(t)=\al_0$ for $\lvert t \rvert \geq \ov{l_j}$, 
where $\al_0$ is a positive constant and $l_j<\ov{l_j} \leq l_j+d_j$. In the following text, 
we take $\al_j(t)$ in the specific form for $j=1,2,3$
For example, for $j=1,2,3$
\begin{equation}
	\al_j(t)=
	\begin{cases}
		0, &\lvert t\rvert \leq l_j,\\
		\ds\al_0\frac{h\Big(\frac{|t|-l_j}{\ov{l_j}-l_j}\Big)}{h\Big(\frac{|t|-l_j}{\ov{l_j}-l_j}\Big)
		+h\Big(\frac{\ov{l_j}-|t|}{\ov{l_j}-l_j}\Big)},\quad& l_j\leq \lvert t\rvert \leq \ov{l_j}, \\
		\al_0,& \lvert t\rvert \geq \ov{l_j},
	\end{cases}
\end{equation}
where $h(s)=\left\{
\begin{aligned}
	&e^{-\frac{1}{s}},\;  &s>0,\\
	& 0,&s\leq 0.
\end{aligned}
\right.$

For $x=(x_1,x_2,x_3)^{\top}\in \R^3$, let $z=\zeta+{\bf i}$ with fixed parameter $\zeta>0$, and define the 
PML medium property as $s_j(x_j)=1+z\al_j(x_j),j=1,2,3.$ Now we introduce the complex stretched coordinate
\begin{equation}\label{complexcoo}
	\wid{x}_j(x_j)=x_j+z\int_0^{x_j} \al_j(t)dt=\int_0^{x_j} [1+z\al_j(t)]dt,
\end{equation}
and it follows obviously
\begin{equation*}
	\frac{d\wid{x}_j}{dx_j}=1+z\al(x_j)=s_j(x_j).
\end{equation*}
Denote by $\na F=\Diag\{s_1(x_1),s_2(x_2),s_3(x_3)\}$ the Jacobi matrix of $F=(\wid{x}_1(x_1),\wid{x}_2(x_2),
\wid{x}_3(x_3))^{\top}.$ Direct calculations show that
\begin{equation}\label{dif_operator_pml}
	\begin{aligned}
		& \wid{\na}\cdot =J^{-1}\na \cdot J(\na F)^{-1},\quad J={\rm det}(\na F), \\
		& \wid{\na}=\left(\frac{1}{s_1(x_1)}\frac{\pa}{\pa x_1},\frac{1}{s_2(x_2)}
		\frac{\pa}{\pa x_2},\frac{1}{s_3(x_3)}\frac{\pa}{\pa x_3}\right).
	\end{aligned}
\end{equation}
Define the complex distance
\begin{equation}\label{complexdistance}
	d(\wid{x},\wid{y})=\left[(\wid{x}_1-\wid{y}_1)^2+(\wid{x}_2-\wid{y}_2)^2
	+(\wid{x}_3-\wid{y}_3)^2\right]^{1/2},
\end{equation}
where $\wid{x}=(\wid{x}_1,\wid{x}_2,\wid{x}_3)^{\top}$ and $\wid{y}=(\wid{y}_1,\wid{y}_2,
\wid{y}_3)^{\top}$ are the complex stretched coordinates defined by \eqref{complexcoo}. 
A direct estimate gives (see also \cite[equation (2.16)]{CXZ2016})
\begin{equation}\label{pmld}
	\Ip d(\wid{x},\wid{y})\geq \frac{\ds\sum_{j=1}^3\left(\lvert x_j-y_j\rvert
	\left\lvert\int_{y_j}^{x_j}\al_j(t)dt\right\rvert+\zeta\left\lvert \int_{y_j}^{x_j}\al_j(t)dt
	\right\rvert^2\right)}{(1+\zeta\al_0)\sqrt{(x_1-y_1)^2+(x_2-y_2)^2+(x_3-y_3)^2}}.
\end{equation}
Define the stretched single- and double-layer potentials
\begin{equation*}
	\wid{\Psi}_{SL}(\bs{p})(x)=\int_{\pa B_1}\wid{\Phi}(x-y,\om)\bs{p}(y)ds(y),
	\quad\bs{p}\in H^{-1/2}(\pa B_1)^4
\end{equation*}
and
\begin{equation*}
	\wid{\Psi}_{DL}(\bs{q})(x)=\int_{\pa B_1}[\mathcal{S}(\pa_y,\nu(y),\om)
	\wid{\Phi}^{\top}(x-y,\om)]^{\top}\bs{q}(y)ds(y),\quad\bs{q}\in H^{1/2}(\pa B_1)^4
\end{equation*}
where the thermoelastic PML modified fundamental solution $\wid{\Phi}(x-y,\om)=(\wid{\Phi}_{ij}(x-y,\om))$
is a $4\times 4$ matrix-valued function with abstract form entries for $i,j=1,2,3$
\begin{align}\label{PMLextensionfs}
	\wid{\Phi}_{ij}(x-y,\om)=&\sum\limits_{l=1}^3\Bigg[ c_{1l}f_{\la_l}(d(\wid{x},\wid{y}))
	+c_{2l}\frac{\wid{x}_i-\wid{y}_i}{d(\wid{x},\wid{y})}f^{'}_{\la_l}(d(\wid{x},\wid{y}))
	+c_{3l}\frac{\wid{x}_j-\wid{y}_j}{d(\wid{x},\wid{y})}f^{'}_{\la_l}(d(\wid{x},\wid{y}))\\ \no
	& +c_{4l}\left(\frac{(\wid{x}_i-\wid{y}_i)(\wid{x}_j-\wid{y}_j)}{d(\wid{x},\wid{y})^2}
	f^{''}_{\la_l}(d(\wid{x},\wid{y}))-\frac{(\wid{x}_i-\wid{y}_i)(\wid{x}_j-\wid{y}_j)}
	{d(\wid{x},\wid{y})^3}f^{'}_{\la_l}(d(\wid{x},\wid{y}))\right) \Bigg],
\end{align}
where $c_{kl}$ are constants and $f_{\la_{l}}(z)=e^{{\bf i}\la_lz}/z$ for $k=1,2,3,4$ and $l=1,2,3$.

Let the PML extension of $E(\bs{f})(x)$ be defined as $\wid{E}(\bs{f})(x)=E(\bs{f})(\wid{x})$, given by
\begin{equation}\label{pmlextension}
	\wid{E}(\bs{f})=\frac{1}{2}\Big[\wid{\Psi}_{SL}(\mathcal{N}\bs{f})-\wid{\Psi}_{DL}(\bs{f})\Big],
\end{equation}
which is obviously a solution to the exterior Dirichlet problem
\begin{equation}\label{eq_pml}
	\begin{aligned}
		\wid{\na}\cdot\wid{\sig}(\wid{\bs{u}})+\rho \om^2 \wid{\bs{u}}-\g \wid{\na}\wid{ p}=0 
		& \quad {\rm in} \;\R^3\ba\ov{B}_1 ,\\
		\wid{\na}\cdot\wid{\na}\wid{ p}+q\wid{p}+{\bf{i}}\om \eta \wid{\na}\cdot\wid{ \bs{u}}=0 
		&\quad {\rm in} \;\R^3\ba \ov{B}_1,\\
		(\wid{\bs{u}}^{\top},\wid{\bs{p}})^{\top}=\bs{f} & \quad {\rm on}\; \pa B_1,
	\end{aligned}
\end{equation}
for given $\bs{f} \in H^{\frac{1}{2}}(\pa B_1)^4$. Here
\begin{equation*}
	\wid{\sig}(\wid{\bs{u}})=2\mu \wid{\epsilon }(\wid{\bs{u}})+\la \Trace(\wid{\epsilon }
	(\wid{\bs{u}}))\mat{I},\quad\wid{\epsilon }(\wid{\bs{u}})=\frac{1}{2}(\wid{\na} \wid{\bs{u}}
	+(\wid{\na}\wid{\bs{u}})^{\top})=\frac{1}{2}(\na \wid{\bs{u}}B^{\top}+B(\na \wid{\bs{u}})^{\top}),
\end{equation*}
where $B=(\na F)^{-1}=\Diag\{s_1(x_1)^{-1},s_2(x_2)^{-1},s_3(x_3)^{-1}\}$.

By using the relations \eqref{dif_operator_pml}, the problem \eqref{eq_pml} becomes
\begin{equation}\label{equa1}
	\begin{aligned}
		\na \cdot( \wid{\sig}(\wid{\bs{u}})A)+\rho \om^2 J \wid{\bs{u}}-\g A \na\wid{ p}=0 
		& \quad {\rm in} \;\R^3\ba\ov{B}_1 ,\\
		\na \cdot (K\na\wid{ p})+qJ \wid{p}+{\bf{i}}\om \eta \na \cdot(A\wid{ \bs{u}})=0 
		&\quad {\rm in} \;\R^3\ba \ov{B}_1,\\
		(\wid{\bs{u}}^{\top},\wid{\bs{p}})^{\top}=\bs{f} & \quad {\rm on}\; \pa B_1,
	\end{aligned}
\end{equation}
where the matrices $A$ and $K$ are defined by
\begin{align}\label{matrix_A}
	A&=J(\na F)^{-1} = \Diag\{s_2(x_2)s_3(x_3),s_1(x_1)s_3(x_3),s_1(x_1)s_2(x_2)\}, 
	\\ \label{matrices_K}
	K&= \ds\Diag\Bigg\{\frac{s_2(x_2)s_3(x_3)}{s_1(x_1)},\frac{s_1(x_1)s_3(x_3)}{s_2(x_2)},
	\frac{s_1(x_1)s_2(x_2)}{s_3(x_3)}\Bigg\}.
\end{align}
To obtain the exponential decay property of thermoelastic PML modified fundamental solution,
we introduce the following lemma.
\begin{lemma}\label{f_lambda}
For each $l=1,2,3$, the $n$-th derivative of $f_{\la_{l}}(z)=e^{{\bf i}\la_lz}/z$ can be expressed 
in the form $f_{\la_{l}}^{(n)}(z)=\ds\frac{P(z)e^{{\bf i}\la_lz}}{z^{n+1}}$, where $P(z)$ is a polynomial
of degree $n$. Furthermore, we have the estimate $\ds\lvert f^{(n)}_{\la_l}(z)\rvert\leq 
\frac{C}{\lvert z \rvert^{n+1}}$ on the unit ball $\{z\in \C:\lvert z\rvert <1\}$ in the complex plane.
\end{lemma}
\begin{proof}
We employ the mathematical induction to prove the Lemma.\\	
For $n=1$, it follows that $ f^{'}_{\la_l}(z)=\ds\frac{i\la_l z e^{{\bf i}\la_l z}-e^{{\bf i}\la_l z}}{z^2}$.\\
For $n=k$, we assume $f_{\la_{l}}^{(k)}(z)=\ds\frac{P(z)e^{{\bf i}\la_lz}}{z^{k+1}}$ with ${\rm deg}[P(z)]=k.$
Taking the derivative of $f_{\la_{l}}^{(k)}(z)$, we obtain
\begin{equation*}
	\begin{aligned}
		f_{\la_{l}}^{(k+1)}(z)&=\frac{[P'(z)e^{{\bf i}\la_lz}+P(z)i\la_le^{{\bf i}\la_lz}]z^{k+1}
		-P(z)e^{{\bf i}\la_lz}(k+1)z^k}{(z^{k+1})^2}\\&=\frac{[P'(z)e^{{\bf i}\la_l}
		+P(z)i\la_le^{{\bf i}\la_lz}]z-P(z)e^{{\bf i}\la_lz}(k+1)}{z^{k+2}}.
	\end{aligned}
\end{equation*}
Following the Maximum Modulus Principle	for analytic function $P(z)e^{{\bf i}\la_lz}$, it holds that
$\lvert P(z)e^{{\bf i}\la_lz}\rvert\leq C$ for $\lvert z\rvert<1$. So $\lvert f^{(n)}_{\la_l}(z)\rvert
\leq \ds\frac{C}{\lvert z \rvert^{n+1}}$. This completes the proof.		
\end{proof}
	
For consistency with physical reality, the real and imaginary parts of the complex wave numbers
$\la_1$ and $\la_2$ must be positive \cite{Cakoni1999} (see also \cite{wang2025} for the rigorous
mathematical proof). This fact can be stated as the following Lemma.
\begin{lemma}
If the characteristic roots $\la_1$ and $\la_2$ of the equations \eqref{chara1} 
satisfy $\Rp \la_j>0$ for $j=1,2$, then $\Ip\la_j>0$ for $j=1,2$.
\end{lemma}
In the following Lemma, we give the estimates on the entries of thermoelastic PML modified fundamental 
solution and its derivatives of all orders.
\begin{lemma}\label{pmlfse}
Denote $\La=\min\{\Rp\la_1,\Rp\la_2,\la_3\}$ and let $d(\wid{x},\wid{y})$ be the complex distance 
which is defined in \eqref{complexdistance}. Under the assumption $\zeta\geq 1$, $ \wid{\Phi}_{ij}(x-y,\om)$ 
satisfies the following estimates
\begin{equation}
	\lvert \wid{\Phi}_{ij}(x-y,\om)\rvert \leq C[(1+\zeta \al_0)^2+\al_0^2]\left(\frac{1}{\lvert x-y\rvert}
	+\frac{1}{\lvert x-y\rvert^3}\right) e^{-\La\cdot\Ip d(\wid{x},\wid{y})},
\end{equation}
and for its derivatives we have
\begin{equation}
	\begin{aligned}
		&\left\lvert\prod_{k=1}^3\left(\frac{\pa}{\pa x_k}\right)^{a_k}
		\prod_{k=1}^3\left(\frac{\pa}{\pa y_k}\right)^{b_k}
		 \wid{\Phi}_{ij}(x-y,\om)\right\rvert \\
		&\leq C[(1+\zeta \al_0)^2+\al_0^2]^{\frac{m_1}{2}}\left(\frac{1}{\lvert x-y\rvert}
		+\frac{1}{\lvert x-y\rvert^{m_2}}\right) e^{-\La\cdot\Ip d(\wid{x},\wid{y})},
	\end{aligned}
\end{equation}
where $a_k,b_k,k=1,2,3$ and $m_1,m_2$ are positive integers.
\end{lemma}
	
\begin{proof}
It follows from \eqref{complexcoo} that the components of the difference of $\wid{x}=(\wid{x}_1,\wid{x}_2,
\wid{x}_3)^{\top}$ and $\wid{y}=(\wid{y}_1,\wid{y}_2,\wid{y}_3)^{\top}$ are
\begin{equation*}
	\begin{aligned}
		\wid{x}_j-\wid{y}_j &=\int_{y_j}^{x_j} [1+z\al_j(t)]dt=(x_j-y_j)[1+z\al_j(t_j)],
	\end{aligned}
\end{equation*}
where $t_j$ lies between $y_j$ and $x_j$ (by the mean value theorem for integrals). 
The square of the complex distance $d(\wid{x},\wid{y})$ is
\begin{equation*}
\begin{aligned}
		&\sum_{j=1}^3(\wid{x}_j-\wid{y}_j)^2=\sum_{j=1}^3(x_j-y_j)^2\Big[1+z\al_j(t_j)\Big]^2.
	\end{aligned}
\end{equation*}	
Its modulus satisfies the following upper bound estimate:
\begin{equation}\label{upper_bound}
	\begin{aligned}
		\lvert d(\wid{x},\wid{y})\rvert^2
		=&\left\lvert \sum_{j=1}^3(x_j-y_j)^2\Big[1+z\al_j(t_j)\Big]^2 \right\rvert\\
		\leq &\sum_{j=1}^3(x_j-y_j)^2\lvert 1+z\al_j(t_j) \rvert^2\\
		=&\sum_{j=1}^3(x_j-y_j)^2\Bigg\{\Big[1+\zeta \al_j(t_j)\Big]^2+\al_j(t_j)^2\Bigg\}\\
		\leq &\Big[(1+\zeta \al_0)^2+\al_0^2\Big]\sum_{j=1}^{3}(x_j-y_j)^2.
	\end{aligned}
\end{equation}
For $\zeta\geq 1$, we have the lower bound estimate
\begin{equation}\label{lower_bound}
	\begin{aligned}
		\lvert d(\wid{x},\wid{y})\rvert^2&\geq \Rp \sum\limits_{j=1}^3(x_j-y_j)^2\Big[1+z\al_j(t_j)\Big]^2\\
		 &=\sum_{j=1}^3(x_j-y_j)^2\Bigg\{\Big[1+\zeta\al_j(t_j)\Big]^2-\al_j(t_j)^2\Bigg\}\\
		 &\geq \sum_{j=1}^{3}(x_j-y_j)^2.
	\end{aligned}
\end{equation}
Therefore, under the assumption $\zeta\geq 1$, a combination of \eqref{upper_bound} and \eqref{lower_bound}
shows that the complex distance $d(\wid{x},\wid{y})$ satisfies
\begin{equation}\label{estimate_dis}
	\lvert x-y\rvert \leq \lvert d(\wid{x},\wid{y})\rvert 
	\leq \sqrt {(1+\zeta \al_0)^2+\al_0^2}\cdot\lvert x-y\rvert,
\end{equation}
and
\begin{equation*}
	\begin{aligned}
		\frac{\lvert \wid{x}_j-\wid{y}_j\rvert}{\lvert d(\wid{x},\wid{y})\rvert }
		&=\frac{\lvert (x_j-y_j)[1+z\al_j(t_j)]\rvert}{\lvert d(\wid{x},\wid{y})\rvert}
		\leq\frac{\lvert x_j-y_j\rvert\sqrt {(1+\zeta \al_0)^2+\al_0^2}}{\lvert d(\wid{x},\wid{y})\rvert}\\
		&\leq \frac{\lvert x-y\rvert\sqrt {(1+\zeta \al_0)^2+\al_0^2}}{\lvert d(\wid{x},\wid{y})\rvert}
		\leq \sqrt {(1+\zeta \al_0)^2+\al_0^2}.
	\end{aligned}
\end{equation*}
It follows from the abstract form \eqref{PMLextensionfs} that
\begin{equation*}
	\begin{aligned}
		\lvert \wid{\Phi}_{ij}(x-y,\om)\rvert 
		\leq & \sum\limits_{l=1}^{3}\bigg\{C_{1l}\lvert f_{\la_l}(d(\wid{x},\wid{y}))\rvert+(C_{2l}+C_{3l})
		\sqrt{(1+\zeta \al_0)^2+\al_0^2}\cdot\lvert f_{\la_l}^{'}(d(\wid{x},\wid{y}))\rvert\\
		&+C_{4l}[(1+\zeta \al_0)^2+\al_0^2] \lvert f_{\la_l}^{''}(d(\wid{x},\wid{y}))\rvert
		+\frac{C_{5l}}{\lvert d(\wid{x},\wid{y})\rvert}[(1+\zeta \al_0)^2+\al_0^2] 
		\lvert f_{\la_l}^{'}(d(\wid{x},\wid{y}))\rvert\bigg\},
	\end{aligned}
\end{equation*}
where $C_{tl}$ are positive constants and $f_{\la_{l}}(z)=e^{{\bf i}\la_lz}/z$ for $t=1,2,3,4,5$ and $l=1,2,3$.
Simple calculations imply that
\begin{equation}\label{exp_pmld}
	\lvert e^{{\bf i}\la_ld(\wid{x},\wid{y})}\rvert 
	= e^{-\Rp\la_l\cdot\Ip d(\wid{x},\wid{y})-\Ip\la_l\cdot\Rp d(\wid{x},\wid{y})}.
\end{equation}
We split the analysis into two cases, depending on whether $\lvert d(\wid{x},\wid{y})\rvert$ 
is less than $1$ or not.

Case 1: $\lvert d(\wid{x},\wid{y})\rvert <1$. Equations \eqref{estimate_dis}-\eqref{exp_pmld}, 
together with Lemma \ref{f_lambda} give
\begin{equation}\label{case1}
	\begin{aligned}
		\lvert \wid{\Phi}_{ij}(x-y,\om)\rvert 
		\leq & \sum\limits_{l=1}^{3}\bigg(C_{1l}\frac{1}{\lvert d(\wid{x},\wid{y})\rvert}
		+(C_{2l}+C_{3l}) \frac{\sqrt {(1+\zeta \al_0)^2+\al_0^2}}{\lvert d(\wid{x},\wid{y})\rvert^2}\\
		&\qquad+(C_{4l}+C_{5l}) \frac{(1+\zeta \al_0)^2+\al_0^2}{\lvert d(\wid{x},\wid{y})\rvert^3}\bigg)\\
		\leq\;& C\sum\limits_{l=1}^{3} \frac{(1+\zeta \al_0)^2+\al_0^2}{\lvert x-y\rvert^3}
		\leq  C\sum\limits_{l=1}^{3}\frac{e^{-\La\cdot\Ip d(\wid{x},\wid{y})}}{e^{-\La}} 
		\frac{(1+\zeta \al_0)^2+\al_0^2}{\lvert x-y\rvert^3}\\
		\leq\;& C\sum\limits_{l=1}^{3}\frac{(1+\zeta \al_0)^2+\al_0^2}{\lvert x-y\rvert^3}
		e^{-\La\cdot\Ip d(\wid{x},\wid{y})},
	\end{aligned}
\end{equation}
where $\La=\min\{\Rp\la_1,\Rp\la_2,\Rp\la_3=\la_3\}$.

Case 2: $\lvert d(\wid{x},\wid{y})\rvert \geq1$. It holds that
\begin{equation*}
	\begin{aligned}
		\lvert \wid{\Phi}_{ij}(x-y,\om)\rvert 
		\leq &\sum\limits_{l=1}^{3} \Bigg\{C_{1l}\frac{e^{-\Rp\la_l\cdot\Ip d(\wid{x},\wid{y})}}
		{\lvert d(\wid{x},\wid{y})\rvert}+(C_{2l}+C_{3l})\sqrt {(1+\zeta \al_0)^2+\al_0^2} 
		\frac{e^{-\Rp\la_l\cdot\Ip d(\wid{x},\wid{y})}}{\lvert d(\wid{x},\wid{y})\rvert}\\
		&+(C_{4l}+C_{5l})[(1+\zeta \al_0)^2+\al_0^2] \frac{e^{-\Rp\la_l
		\cdot\Ip d(\wid{x},\wid{y})}}{\lvert d(\wid{x},\wid{y})\rvert}\Bigg\}\\
		\leq\;& C\sum\limits_{l=1}^{3} [(1+\zeta \al_0)^2+\al_0^2] 
		\frac{e^{-\La\cdot\Ip d(\wid{x},\wid{y})}}{\lvert x-y\rvert},
	\end{aligned}
\end{equation*}
This, combined with \eqref{case1} leads to
\begin{equation*}
	\begin{aligned}
		&\lvert \wid{\Phi}_{ij}(x-y,\om)\rvert \leq C[(1+\zeta \al_0)^2+\al_0^2]
		\Big(\frac{1}{\lvert x-y\rvert}+\frac{1}{\lvert x-y\rvert^3}\Big) e^{-\La\cdot\Ip d(\wid{x},\wid{y})}.
	\end{aligned}
\end{equation*}
The modules of all the derivatives of $ \wid{\Phi}_{ij}(x-y,\om)$ can be similarly estimated. 
The proof is thus complete.
\end{proof}

\section{The PML equation in truncated domain and layer}
\setcounter{equation}{0}
In this section we shall show that the PML system in the truncated domain
\begin{subequations}
	\begin{align} \label{tPML1}
	\na \cdot( \wid{\sig}(\wid{\bs{u}})A)+\rho \om^2 J \wid{\bs{u}}-\g A \na\wid{ p}
	={\bf Q} \quad &\text{in } \Om_2, \\\label{tPML2}
	\na \cdot (K\na\wid{ p})+qJ \wid{p}+{\bf{i}}\om \eta \na \cdot(A\wid{ \bs{u}})=0
	\quad &\text{in } \Om_2,  \\\label{tPML3}
	(\wid{\bs{u}}^{\top},\wid{p})^{\top}=0 \quad &\text{on } \pa\Om, \\\label{tPML4}
	(\wid{\bs{u}}^{\top},\wid{p})^{\top}=0 \quad &\text{on } \pa B_2,
	\end{align}
\end{subequations}
has a unique weak solution in the space $H_0^1( \Om_2)$ under appropriate constraints 
on the thermoelastic and PML parameters. To this end, we introduce the sesquilinear forms
\begin{align}
	\mathcal{B}_{D}(\Phi,\Psi)=&\int_{D}\Big[\wid{\sig}(\wid{\bs{u}})A \colon \na\ov{\bs{u}}'
	-\rho \om^2J\wid{\bs{u}}\cdot \ov {\bs{u}}'-\g \wid{p}\na\cdot (A\ov{\bs{u}}') \no\\
	\label{biform_B}
	&+\na \ov{p}'\cdot(K\na\wid{p})-qJ\wid{p}\ov{p}'-{\bf{i}}\om \eta \ov{p}'\na \cdot (A\wid{\bs{u}})\Big]dx,\\ 
	\label{biform_A}
	\mathcal{A}_{D}(\Phi,\Psi)=&\int_{D}\left[\wid{\sig}(\wid{\bs{u}})A \colon \na\ov{ \bs{u}}'
	+\na \ov{p}'\cdot(K\na\wid{p})\right]dx,
\end{align}
where $\Phi=(\wid{\bs{u}},\wid{p})$, $\Psi=(\bs{u}',p')$ and $D$ is a bounded domain with piecewise smooth boundary.

To get the coercivity of the sesquilinear form $\mathcal{A}_{D}$, we need the following elementary lemma.
\begin{lemma}{\rm \cite[Lemma 3.1]{CXZ2016}.}\label{lowbound_s}
Assume that $\zeta \geq \sqrt{(\la+2\mu)/\mu}$ and $D$ is a bounded domain with piecewise smooth boundary. 
Then for $j=1,2,3$, we have
\begin{equation*}
	\begin{aligned}
		&\Big(1+\frac{\mu}{\la+\mu}\Big)\Rp \frac{s_1(x_1)s_2(x_2)s_3(x_3)}{s_j(x_j)^2}\\{}
		&\geq \frac{[1+\zeta \al_1(x_1)][1+\zeta \al_2(x_2)][1+\zeta \al_3(x_3)]}{[1+\zeta \al_j(x_j)]^2}
		+\frac{\mu}{\la+\mu}\frac{1}{\lvert s_j(x_j) \rvert ^2}.
	\end{aligned}
\end{equation*}
\end{lemma}
The following lemma extends arguments from \cite[Lemma 3.3]{CXZ2016}.
\begin{lemma}[Ellipticity lemma]\label{ellip}
Assume that $\zeta \geq \sqrt{(\la+2\mu)/\mu}$ and $D$ is a bounded domain with piecewise smooth boundary.
Then for any $\wid{\bs{u}} \in H_{0}^1(D)^3$ and $\wid{p} \in H_{0}^1(D)$, we have
\begin{equation}\label{elliptic_u}
	\Rp \int_{D}\wid{\sig}(\wid{\bs{u}})A \colon \na\ov{ \wid{\bs{u}}}dx
	\geq C_1\Vert \na \wid{\bs{u}}\Vert_{L^2(D)^{3\times 3}}^2,
\end{equation}
and
\begin{equation}\label{elliptic_p}
	\Rp \int_{D}\na \ov{\wid{p}}\cdot(K\na\wid{p})dx\geq C_2\Vert \na \wid{p}\Vert_{L^2(D)^3}^2,
\end{equation}
where constants $C_1,C_2$ depend on $\al_0$ and the size of $D$, growing at most polynomially in both.
\end{lemma}
\begin{remark}
For $\bs{u}=(u_1,u_2,u_3)^{\top}$, the usual $H^1$-norm $\|\cdot\|_{H^1(D)^3}$ has the representation
\begin{equation*}
	\|\bs{u}\|_{H^1(D)^3}=\left(\|\bs{u}\|^2_{L^2(D)^3}+\|\na\bs{u}\|^2_{L^2(D)^{3\times 3}}\right)^{1/2}
\end{equation*}
with the Frobenius norm
\begin{equation*}
	\|\na\bs{u}\|_{L^2(D)^{3\times 3}}=\left(\sum_{j=1}^3\int_{D}|\na u_j|^2dx\right)^{1/2}.
\end{equation*}
\end{remark}
\begin{proof}
The proof of inequality \eqref{elliptic_u} is a direct consequence of \cite[Lemma 3.3]{CXZ2016}. 
Recalling the definition of matrix $K$ (see \eqref{matrices_K}), the inequality \eqref{elliptic_p} follows from
\begin{equation*}
	\begin{aligned}
		\Rp \int_{D}\na \ov{\wid{p}}\cdot(K\na\wid{p})dx=&\int_{D}\Rp \frac{s_2(x_2)s_3(x_3)}{s_1(x_1)}
		\Big\lvert  \frac{\pa \wid{p}}{\pa x_1}\Big\rvert^2dx\\
		&+\int_{D}\Rp \frac{s_1(x_1)s_3(x_3)}{s_2(x_2)}\Big\rvert\frac{\pa \wid{p}}{\pa x_2}\Big\lvert^2dx\\
		&+\int_{D}\Rp \frac{s_1(x_1)s_2(x_2)}{s_3(x_3)}\Big\lvert\frac{\pa \wid{p}}{\pa x_3}\Big\rvert^2dx,
	\end{aligned}
\end{equation*}
and the boundedness of $D$ and Lemma \ref{lowbound_s}. The proof is thus complete.
\end{proof}

Now we are in the position to present the main result of this section.
\begin{theorem}\label{maintheorem}
Assume $D=\Om_2$ or $D$ has smooth boundaries. If the parameters satisfy
\begin{equation}\label{pml_pc}
	\begin{cases}
		\ds\zeta \geq \sqrt{(\la+2\mu)/\mu},\\
		\ds\frac{\rho \g^2}{\eta^2}\geq 1,\\
		\ds\zeta^2-1\geq 2\g \zeta ,\\
		\ds\zeta \geq \sqrt{3},\\
		\ds\al_0<\frac{\la+\mu}{2\g(\la+2\mu)},
	\end{cases}
\end{equation}
then there exists a constant $C>0$ that
\begin{equation*}
	\sup_{\Psi\in H_0^1(D)^4\ba\{0\}}\frac{\lvert \mathcal{B}_{D}(\Phi,\Psi) \rvert}
	{\Vert \Psi \Vert_{H_0^1(D)^4}} \geq C \Vert \Phi \Vert_{H_0^1(D)^4},\quad\forall\;\Phi\in H_0^1(D)^4
\end{equation*}
except possibly for a discrete set of frequencies $\om$.
\begin{proof}
We only need to show that the variational problem
\begin{equation}\label{pml_variation}
	\mathcal{B}_{D}(\Phi,\Psi)=-\int_{D}Q_1\cdot \ov{\Psi}dx,\quad \forall\;\Psi \in H_0^1(D)^4
\end{equation}
has a unique weak solution in $H_0^1(D)^4$, where $Q_1$ is in the dual space of the $H_0^1(D)^4$.

For $\Phi=(\wid{\bs{u}}^{\top},\wid{p})^{\top}$, $\Psi=(\bs{u}'^{\top},p')^{\top}$ with 
$\wid{\bs{u}}=(\wid{u}_1,\wid{u}_2,\wid{u}_3)^{\top}$ and $\bs{u}'=(u'_1,u'_2,u'_3)^{\top}$, let
\begin{equation*}
	\mathcal{H}_{D}(\Phi,\Psi):=\sum_{i=1}^{3}(\na \wid{u}_i,\na u'_i)+(\na \wid{p},\na p')+(\Phi,\Psi),
\end{equation*}
where  $(\cdot,\cdot)$ denotes the $L^2$-inner product, defined by $(\bs{u},\bs{v})
=\ds\int_{D}\bs{u}\cdot \ov{\bs{v}}dx$ for complex-valued vector functions $\bs{u},\bs{v}$.

Define operators $\mathcal{P}_1(\om),\mathcal{P}_2(\om): H_0^1(D)^4 \to H_0^1(D)^4$ by
\begin{align}
	\mathcal{H}_D(\mathcal{P}_1\Phi,\Psi) &= \mathcal{A}_D(\Phi,\Psi) + (\Phi,\Psi), \label{e1} \\
	\mathcal{H}_D(\mathcal{P}_2\Phi,\Psi) &= (\Phi,\Psi) + \int_D\Bigl[
	\rho\om^2 J\wid{\bs{u}}\cdot\ov{\bs{u}}'
	+ \g\wid{p}\na\cdot(A\ov{\bs{u}}') \notag \\
	&\quad + qJ\wid{p}\,\ov{p}'
	+ {\bf{i}}\om\eta\ov{p}'\na\cdot(A\wid{\bs{u}})\Bigr]dx, \label{T}
\end{align}
Let
\begin{equation*}
	f^{\Phi}(\Psi)=\int_{D}\Big[\rho \om^2J\wid{\bs{u}}\cdot \ov {\bs{u}}'+\g \wid{p}\na
	\cdot (A\ov{\bs{u}}')+qJ\wid{p}\ov{p}'+{\bf{i}}\om \eta \ov{p}'\na \cdot (A\wid{\bs{u}})\Big]dx
\end{equation*}
and $g^{\Phi}(\Psi)=\ov{f^{\Phi}(\Psi)}.$ The $g^{\Phi}(\Psi)$ is a continuous linear functional 
on the space $H_0^1(D)^4$. Since $H_0^1(D)^4$ is the subspace of $L^2(D)^4$, it follows from 
the Hahn-Banach theorem that the linear functional $g^{\Phi}$ can be extended to a linear functional
$g_1^{\Phi}$ on the whole space $L^2(D)^4$ satisfying
\begin{equation*}
	g_1^{\Phi}(\Psi)=g^{\Phi}(\Psi) ,\quad \forall\;\Psi \in H_0^1(D)^4.
\end{equation*}
Let $f_1^{\Phi}(\Psi)=\ov{g_1^{\Phi}(\Psi)}$, which is an anti-linear functional on the space $L^2(D)^4$.
By the Riesz representation theorem, there exists a unique $u_{\Phi} \in L^2(D)^4$ such that 
$f_1^{\Phi}(\Psi)=(u_{\Phi},\Psi),\;\forall\;\Psi \in L^2(D)^4$. So we can rewrite the right of equation \eqref{T} as
\begin{equation}
	\mathcal{H}_{D}(\mathcal{P}_2\Phi,\Psi)=(\Phi,\Psi)+(u_{\Phi},\Psi)
\end{equation}
By the lemma \ref{ellip}, the sesquilinear form in the right of equations \eqref{e1} is strongly coercive. 
We have that the $\mathcal{P}_1^{-1}$ exists and is continuous by using the Lax-Milgram lemma. 
In the equations \eqref{T}, the regularity theory for the strongly elliptic systems with constant coefficients
in the domains with corners \cite{Savare1998} implies that
\begin{equation*}
	\exists\;s >1, \text{such that}\;\mathcal{P}_2\Phi \in H^s(D)^4, \;{\rm if}\;\Phi+u_{\Phi}\in L^2(D)^4,
\end{equation*}
we know that $\mathcal{P}_2$ is compact from that $H^s(D)^4\cap H_0^1(D)^4$ is compactly embedded in $H_0^1(D)^4$.
The following linear partial differential system
\begin{equation}
	\mathcal{H}_{D}(G,\Psi)=-\int_{D}Q_1\cdot \ov{\Psi}dx,\quad \forall\;\Psi \in H_0^1(D)^4
\end{equation}
always has a unique solution in the $H_0^1(D)^4$. So, finding a solution to equation\eqref{equa1} 
in the $H_0^1(D)^4 $ is equivalent to seeking $\Phi \in H_0^1(D)^4$ such that
\begin{equation}\label{equi}
	\mathcal{P}_1\Phi-\mathcal{P}_2\Phi=G.
\end{equation}
The $\mathcal{P}_1^{-1}\mathcal{P}_2$ is compact since $\mathcal{P}_1^{-1}$ is continuous. 
To study the solvalbilty of operator equation \eqref{equi} by the analytic Fredholm theorem, 
we need to seek a frequency such that there exists a unique solution to \eqref{pml_variation}.

Taking frequency $\om=\frac{\g}{\eta}{\bf i}$ into the sesquilinear \eqref{biform_B} gives
\begin{align} \no
	\mathcal{B}_{D}(\Phi,\Phi)
	&=\int_{D}\Big[\wid{\sig}(\wid{\bs{u}})A \colon \na\ov{\wid{\bs{u}}}
	+\rho \frac{\g^2}{\eta^2}J\wid{\bs{u}}\cdot \ov{\wid{\bs{u}}}+\na \ov{\wid{p}}\cdot(K\na\wid{p})
	+\frac{\g}{\eta \kappa}J\wid{p}\ov{\wid{p}} \\ \label{biform_B_2}
	&\qquad+\g \na \wid{p}\cdot (A\ov{\wid{\bs{u}}})-\g\na \ov{\wid{p}} \cdot (A\wid{\bs{u}})\Big]dx \\ \no
	&= \mathcal{A}_{D}(\Phi,\Phi) +\int_{D}\Big(\rho \frac{\g^2}{\eta^2}J\wid{\bs{u}}\cdot 
	\ov {\wid{\bs{u}}}+\frac{\g}{\eta \kappa}J\wid{p}\ov{\wid{p}}\Big)dx 
	+ \int_{D}\Big[\g \na \wid{p}\cdot (A\ov {\wid{\bs{u}}})
	-\g\na \ov{\wid{p}} \cdot (A\wid{\bs{u}})\Big]dx,
\end{align}
where the sesquilinear $\mathcal{A}_{D}(\Phi,\Phi)$ is defined by \eqref{biform_A}. 
Direct calculation implies
\begin{align}\label{reJ}
	\Rp J &= \prod_{j=1}^3[1+\zeta \al_j(x_j)] - \sum_{j=1}^3\frac{\al_1(x_1)\al_2(x_2)
	\al_3(x_3)}{\al_j(x_j)}-3\zeta\al_1(x_1)\al_2(x_2)\al_3(x_3)\\ \no
	&=1+\zeta\sum_{j=1}^{3}\al_j(x_j) + (\zeta^2-1)\sum_{j=1}^3\frac{\al_1(x_1)\al_2(x_2)
	\al_3(x_3)}{\al_j(x_j)}+(\zeta^3-3\zeta)\al_1(x_1)\al_2(x_2)\al_3(x_3).
\end{align}
For given vector $\wid{\bs{u}}=(\wid{u}_1,\wid{u}_2,\wid{u}_3)^{\top}$ and 
recalling the definition of matrix $A$ (see \eqref{matrix_A}), we obtain
\begin{equation*}
	\begin{aligned}
		\Rp\int_{D}\g \na \wid{p}\cdot (A\ov {\wid{\bs{u}}})dx
		&=\g \Rp\int_{D}\left(s_2s_3\frac{\pa \wid{p}}{\pa x_1}\ov {\wid{u}_1} 
		+s_1s_3\frac{\pa \wid{p}}{\pa x_2}\ov {\wid{u}_2}
		+s_1s_2\frac{\pa \wid{p}}{\pa x_3}\ov {\wid{u}_3}\right)dx,\\
		\Rp\int_{D}\g\na \ov{\wid{p}} \cdot (A\wid{\bs{u}})dx 
		&=\g \Rp\int_{D}\left(s_2s_3\ov{\frac{\pa \wid{p}}{\pa x_1}}\wid{u}_1 
		+s_1s_3\ov{\frac{\pa \wid{p}}{\pa x_2}}\wid{u}_2
		+s_1s_2\ov{\frac{\pa \wid{p}}{\pa x_3}}\wid{u}_3\right)dx.
	\end{aligned}
\end{equation*}
This, combining the relation $s_j=1+(\zeta+{\bf i})\al_j$ and the basic identity
\begin{equation*}
	\frac{\pa \wid{p}}{\pa x_j}\ov {\wid{u}_j}-\ov{\frac{\pa \wid{p}}{\pa x_j}}\wid{u}_j 
	= 2{\bf i}\Ip \left(\frac{\pa \wid{p}}{\pa x_j}\ov {\wid{u}_j}\right),j=1,2,3
\end{equation*}
gives
\begin{align}\no
	&\Rp\int_{D}\Big[\g \na \wid{p}\cdot (A\ov {\wid{\bs{u}}})
	-\g\na \ov{\wid{p}} \cdot (A\wid{\bs{u}})\Big]dx\\\no
	=&-2\g \int_{D}\Bigg\{(\al_2+\al_3+2\zeta \al_2\al_3)
	\Ip\left(\frac{\pa \wid{p}}{\pa x_1}\ov {\wid{u}_1}\right)+(\al_1+\al_3
	+2\zeta \al_1\al_3)\Ip\left(\frac{\pa \wid{p}}{\pa x_2}\ov {\wid{u}_2}\right)\\\label{B_1st_order1}
	&+(\al_1+\al_2+2\zeta \al_1\al_2)\Ip\left(\frac{\pa \wid{p}}{\pa x_3}\ov {\wid{u}_3}\right)\Bigg\}dx.
\end{align}
Noting that Cauchy-Schwartz inequality implies
\begin{equation*}
	-2\Ip\left(\frac{\pa \wid{p}}{\pa x_j}\ov {\wid{u}_j}\right) 
	\geq -\Big(\Big\lvert \frac{\pa \wid{p}}{\pa x_j}\Big\rvert^2
	+\Big\lvert \ov {\wid{u}_j}\Big\rvert ^2\Big), j=1,2,3,
\end{equation*}
together with \eqref{B_1st_order1}, we arrive at the following estimate
\begin{align}\no
	&\Rp\int_{D}\Big[\g \na \wid{p}\cdot (A\ov {\wid{\bs{u}}})
	-\g\na \ov{\wid{p}} \cdot (A\wid{\bs{u}})\Big]dx\\\no
	\geq& -\g \int_{D}\Bigg\{(\al_2+\al_3+2\zeta \al_2\al_3) \Big(\Big\lvert \frac{\pa \wid{p}}
	{\pa x_1}\Big\rvert^2+ \Big\lvert \ov {\wid{u}_1}\Big\rvert ^2\Big)+(\al_1+\al_3
	+2\zeta \al_1\al_3)\Big(\Big\lvert \frac{\pa \wid{p}}{\pa x_2}\Big\rvert^2
	+ \Big\lvert \ov {\wid{u}_2}\Big\rvert ^2\Big)\\\no
	&+(\al_1+\al_2+2\zeta \al_1\al_2)\Big(\Big\lvert \frac{\pa \wid{p}}{\pa x_3}\Big\rvert^2
	+ \Big\lvert \ov {\wid{u}_3}\Big\rvert ^2\Big)\Bigg\}dx\\\no
	\geq &-2\g\int_{D}(\al_1+\al_2+\al_3+\zeta \al_2\al_3 +\zeta \al_1\al_3
	+\zeta \al_1\al_2) |\wid{\bs{u}}|^2dx\\\no
	&-\g\int_{D}(\al_2+\al_3+2\zeta \al_2\al_3 )\Big\lvert \frac{\pa \wid{p}}{\pa x_1}\Big\rvert^2 dx
	-\g\int_{D}(\al_1+\al_3+2\zeta \al_1\al_3\Big ]\Big\lvert \frac{\pa \wid{p}}{\pa x_2}\Big\rvert^2dx
	\\\label{biform_B_1st_order}
	&-\g\int_{D}(\al_1+\al_2+2\zeta \al_1\al_2 )\Big\lvert \frac{\pa \wid{p}}{\pa x_3}\Big\rvert^2dx.
\end{align}
It follows from Lemmas \ref{lowbound_s}-\ref{ellip} that
\begin{equation}\label{biform_B_2ed_order}
	\begin{aligned}
		\Rp \mathcal{A}_{D}(\Phi,\Phi)
		\geq\;& C_1\Vert \na \wid{\bs{u}}\Vert_{L^2(D)^{3\times 3}}^2
		+\int_{D}\Big(\frac{\la+\mu}{\la+2\mu}\Big)\frac{(1+\zeta \al_2)(1+\zeta \al_3)}
		{(1+\zeta \al_1)}\Big\lvert \frac{\pa \wid{p}}{\pa x_1}\Big\rvert^2dx
		\\&+\int_{D}\Big(\frac{\la+\mu}{\la+2\mu}\Big)\frac{(1+\zeta \al_1)(1+\zeta \al_3)}
		{(1+\zeta \al_2)}\Big\lvert \frac{\pa \wid{p}}{\pa x_2}\Big\rvert^2dx
		\\&+\int_{D}\Big(\frac{\la+\mu}{\la+2\mu}\Big)\frac{(1+\zeta \al_1)(1+\zeta \al_2)}
		{(1+\zeta \al_3)}\Big\lvert \frac{\pa \wid{p}}{\pa x_3}\Big\rvert^2dx.
	\end{aligned}
\end{equation}
Under the assumption $\ds\frac{\rho \g^2}{\eta^2}\geq 1$, we have
\begin{equation}\label{biform_B_0_order}
	\begin{aligned}
		\Rp \int_{D}\Big(\rho \frac{\g^2}{\eta^2}J\wid{\bs{u}}\cdot \ov {\wid{\bs{u}}}
		+\frac{\g}{\eta \kappa}J\wid{p}\ov{\wid{p}}\Big)dx
		\geq \int_{D}\Rp J |\wid{\bs{u}}|^2 dx+ \frac{\g}{\eta\kappa}\Vert \wid{p}\Vert_{L^2(D)}^2.
	\end{aligned}
\end{equation}
We require the following three inequalities to hold:
\begin{equation*}
	\begin{aligned}
		&\left(\frac{\la+\mu}{\la+2\mu}\right)\frac{(1+\zeta \al_2)
		(1+\zeta \al_3)}{(1+\zeta \al_1)}>\g(\al_2+\al_3+2\zeta \al_2\al_3 ), \\
		&\left(\frac{\la+\mu}{\la+2\mu}\right)\frac{(1+\zeta \al_1)
		(1+\zeta \al_3)}{(1+\zeta \al_2)}>\g(\al_1+\al_3+2\zeta \al_1\al_3 ),\\
		&\left(\frac{\la+\mu}{\la+2\mu}\right)\frac{(1+\zeta \al_1)
		(1+\zeta \al_2)}{(1+\zeta \al_3)}>\g(\al_1+\al_2+2\zeta \al_1\al_2 ),
	\end{aligned}
\end{equation*}
which can be expanded as
\begin{equation*}
	\begin{aligned}
		&\left(\frac{\la+\mu}{\la+2\mu}\right)(1+\zeta \al_2+\zeta \al_3+\zeta^2\al_2\al_3)\\
		>\;&\g\left[ \al_2+ \al_3+\zeta \al_2(\al_1+ \al_3)+\zeta \al_3(\al_1+ \al_2)
		+2\zeta^2\al_1\al_2\al_3\right],\\
		&\left(\frac{\la+\mu}{\la+2\mu}\right)(1+\zeta \al_1+\zeta \al_3+\zeta^2\al_1\al_3)\\
		>\;&\g\left[ \al_1+ \al_3+\zeta \al_1(\al_2+ \al_3)+\zeta \al_3(\al_1+ \al_2)
		+2\zeta^2\al_1\al_2\al_3\right],\\
		&\left(\frac{\la+\mu}{\la+2\mu}\right)(1+\zeta \al_1+\zeta \al_2+\zeta^2\al_1\al_2)\\
		>\;&\g\left[ \al_1+ \al_2+\zeta \al_1(\al_2+ \al_3)+\zeta \al_2(\al_1+ \al_3)
		+2\zeta^2\al_1\al_2\al_3\right].
	\end{aligned}
\end{equation*}
Since functions $\al_j(x_j)\leq \al_0,\;j=1,2,3$, we need the condition
\begin{equation*}
2\g\frac{\la+2\mu}{\la+\mu}\al_0<1 \Longleftrightarrow \al_0<\frac{\la+\mu}{2\g(\la+2\mu)}
\end{equation*}
to meet the requirements. Therefore, we have the estimate
\begin{equation}\label{estimate_C2}
	\begin{aligned}
		&\int_{D}\left[\left(\frac{\la+\mu}{\la+2\mu}\right)\frac{(1+\zeta \al_2)
		(1+\zeta \al_3)}{(1+\zeta \al_1)}-\g(\al_2+\al_3+2\zeta \al_2\al_3 )\right]
		\Big\lvert \frac{\pa \wid{p}}{\pa x_1}\Big\rvert^2dx
		\\+&\int_{D}\left[\left(\frac{\la+\mu}{\la+2\mu}\right)\frac{(1+\zeta \al_1)
		(1+\zeta \al_3)}{(1+\zeta \al_2)}-\g(\al_1+\al_3+2\zeta \al_1\al_3)\right]
		\Big\lvert \frac{\pa \wid{p}}{\pa x_2}\Big\rvert^2dx
		\\+&\int_{D}\left[\left(\frac{\la+\mu}{\la+2\mu}\right)\frac{(1+\zeta \al_1)
		(1+\zeta \al_2)}{(1+\zeta \al_3)}-\g(\al_1+\al_2+2\zeta \al_1\al_2 )\right]
		\Big\lvert \frac{\pa \wid{p}}{\pa x_3}\Big\rvert^2dx\\
		\geq\;& C_2\left(\Big\Vert \frac{\pa \wid{p}}{\pa x_1}\Big\Vert_{L^2(D)}^2
		+\Big\Vert \frac{\pa \wid{p}}{\pa x_2}\Big\Vert_{L^2(D)}^2
		+\Big\Vert \frac{\pa \wid{p}}{\pa x_3}\Big\Vert_{L^2(D)}^2\right)
		=C_2\Vert\na\wid{p}\Vert_{L^2(D)^3}^2.
	\end{aligned}
\end{equation}
Under the assumptions $\zeta^2-1\geq 2\g \zeta$ and $\zeta \geq\sqrt{3}$, we obtain from \eqref{reJ} that
\begin{equation*}
	\begin{aligned}
		\int_{D}\Rp J \lvert \wid{\bs{u}}\rvert^2 dx - 2\g\int_{D}(\al_1+\al_2+\al_3
		+\zeta \al_2\al_3 +\zeta \al_1\al_3+\zeta \al_1\al_2)\lvert \wid{\bs{u}}\rvert^2dx
		\geq C_3 \Vert \wid{\bs{u}}\Vert_{L^2(D)^3}^2.
	\end{aligned}
\end{equation*}
This, combined with \eqref{biform_B_2} and \eqref{biform_B_2ed_order}-\eqref{estimate_C2} 
implies for any $\Phi \in H_0^1(D)^4$
\begin{equation*}
	\begin{aligned}
		\Rp \mathcal{B}_{D}(\Phi,\Phi)\geq\;& C_1\Vert \na \wid{\bs{u}}\Vert_{L^2(D)^{3\times 3}}^2
		+C_2\Vert\na\wid{p}\Vert_{L^2(D)^3}^2+C_3\Vert  \wid{\bs{u}}\Vert_{L^2(D)^3}^2
		+ \frac{\g}{\eta\kappa}\Vert  \wid{p}\Vert_{L^2(D)}^2,
	\end{aligned}
\end{equation*}
Thus for the specific frequency $\ds\om=\frac{\g}{\eta}{\bf i}$, under the assumptions \eqref{pml_pc}, 
the Lax-Milgram lemma shows that a unique solution exists for the variational problem
\begin{equation*}
	\mathcal{B}_{D}(G,\Psi)=-\int_{D}Q_1\cdot \ov{\Psi}dx,\quad \forall\;\Psi \in H_0^1(D)^4
\end{equation*}
Finally, the analytic Fredholm theorem implies the $(I-\mathcal{P}_1^{-1}\mathcal{P}_2)^{-1}$ exists 
for equation \eqref{equi} except possibly for a discrete set of frequencies $\om$.
\end{proof}

\end{theorem}

From the proof for the case $D=\Om_2$ in Theorem \ref{maintheorem}, we can obtain the following existence 
and uniqueness result for the truncated PML problem \eqref{tPML1}-\eqref{tPML4}.
\begin{corollary}
Under the conditions \eqref{maintheorem}, the truncated PML problem \eqref{tPML1}-\eqref{tPML4} admits 
a unique solution in the space $H_0^1(\Om_2)^4$, except possibly for a discrete set of frequencies $\om$.
\end{corollary}

In what follows, we consider a PDE system in the PML layer for the convergence analysis of the PML method:
\begin{subequations}
	\begin{align}\label{PMLlayerstabilityequa1}
		\na \cdot( \wid{\sig}(\wid{\bs{u}})A)+\rho \om^2 J \wid{\bs{u}}-\g A \na\wid{ p}=0 
		\quad &{\rm in} \;\Om_{\PML} , \\ \label{PMLlayerstabilityequa22}
		\na \cdot (K\na\wid{ p})+qJ \wid{p}+{\bf{i}}\om \eta \na \cdot(A\wid{ \bs{u}})=0 
		\quad &{\rm in} \;\Om_{\PML},\\\label{PMLlayerconditionouter}
		(\wid{\bs{u}}^{\top},\wid{p})^{\top} =F \quad &{\rm on} \; \pa B_2,
		\\ \label{PMLlayerconditioninner}
		(\wid{\bs{u}}^{\top},\wid{p})^{\top}=0 \quad &{\rm on} \; \pa B_1.
	\end{align}
\end{subequations}
Then the variational formulation of \eqref{PMLlayerstabilityequa1}-\eqref{PMLlayerconditioninner} 
is as follows: given $F\in H^{1/2}(\pa B_2)^4$, find $(\wid{\bs{u}}^{\top},\wid{p})^{\top}\in 
H_{\pa B_1}^1(\Om_{\PML})^4=:\{\Phi\in H^1(\Om_{\PML})^4: \Phi=0\;{\rm on}\;\pa B_1\}$ such that
\begin{equation}\label{varform_BPML}
	\mathcal{B}_{\Om_\PML}((\wid{\bs{u}}^{\top},\wid{p})^{\top},\Psi)=0, \quad\forall\;\Psi\in H_0^1(\Om_{\PML})^4.
\end{equation}

\begin{theorem}\label{pmllayer}
Let $D=\Om_{\PML}$ and assume the parameters satisfy the constraints \eqref{pml_pc}. 
For given $F\in H^{1/2}(\pa B_2)^4$, the system \eqref{PMLlayerstabilityequa1}-\eqref{PMLlayerconditioninner}
admits a unique weak solution $(\wid{\bs{u}}^{\top},\wid{p})^{\top}\in H^{1}_{\pa B_1}(\Om_{\PML})^4$ 
such that $(\wid{\bs{u}}^{\top},\wid{p})^{\top}=F$ on $\pa B_2$ except possibly for a discrete set of 
frequencies $\om$. Moreover,  the following stability estimate holds for the system 
\eqref {PMLlayerstabilityequa1}-\eqref {PMLlayerconditioninner}
\begin{equation}\label{stability_pmllayer}
	\Vert (\wid{\bs{u}}^{\top},\wid{p})^{\top} \Vert_{H^1(\Om_{\PML})^4}\leq C(d)\Vert F \Vert_{H^{1/2}(\pa B_2)^4},
\end{equation}
where the constant $C(d)$ depends on the thickness $d$ and grows at most polynomially in $d$.
\begin{proof}
Firstly, we prove the existence and uniqueness of solution to the problem 
\eqref{PMLlayerstabilityequa1}-\eqref{PMLlayerconditioninner}.
Let $\chi \in C_0^{\infty}(\R^3)$ be a cut-off function such that $\chi=0$
on $B_{d/2}=\{(x_1,x_2,x_3)^{\top} \in \R^3:\lvert x_j \rvert <l_j+d_j/2,j=1,2,3\}$ 
and $\chi=1$ near the boundary $\pa B_2$.

Assume $(\wid{\bs{u}}^{\top},\wid{p})^{\top}\in H^{1}(\Om_{\PML})^4$ is the solution of 
the variational problem \eqref{varform_BPML}, then $(\wid{U}^{\top},\wid{P})^{\top}
=(\wid{\bs{u}}^{\top},\wid{p})^{\top}-\chi F\in H_0^{1}(\Om_{\PML})^4$ satisfies 
the following variational formulation
\begin{equation}\label{varform_PML}
	\mathcal{B}_{\Om_{\PML}}((\wid{U}^{\top},\wid{P})^{\top},\Psi) 
	= -\mathcal{B}_{\Om_{\PML}}(\chi F,\Psi),\quad \Psi\in H_0^{1}(\Om_{\PML})^4
\end{equation}
where $\mathcal{B}_{\Om_{\PML}}(\cdot,\cdot)$ is defined in \eqref{biform_B}. 
Similar discussion with the proof of Theorem \ref{maintheorem} shows that the 
variational problem \eqref{varform_PML} has a unique weak solution in $H_0^1(\Om_\PML)^4$ 
except possibly for a discrete set of frequencies $\om$ under the constraints \eqref{pml_pc}. 
Consequently, the inf-sup condition holds
\begin{equation}\label{infsup_pmllayer}
	\sup_{\Psi \in H_0^1(\Om_{\PML})^4\ba\{0\}}\frac{\lvert \mathcal{B}_{\Om_{\PML}}(\Phi,\Psi) \rvert}
	{\Vert \Psi \Vert_{H_0^1(\Om_{\PML})^4}} \geq C \Vert \Phi \Vert_{H^1(\Om_{\PML})^4},
	\quad\forall\;\Phi \in H_0^1(\Om_{\PML})^4.
\end{equation}
The desired existence and uniqueness result of solution $(\wid{\bs{u}}^{\top},\wid{p})^{\top}$ 
for the problem \eqref{PMLlayerstabilityequa1}-\eqref{PMLlayerconditioninner} is thus obtained.

We now establish the stability estimate for the problem \eqref{PMLlayerstabilityequa1}-\eqref{PMLlayerconditioninner}.
For the variational problem \eqref{varform_PML}, it follows from the inf-sup condition \eqref{infsup_pmllayer} and 
the continuity of $\mathcal{B}_{\Om_{\PML}}(\cdot,\cdot)$ that
\begin{equation*}
	\begin{aligned}
		\Vert (\wid{U}^{\top},\wid{P})^{\top} \Vert_{H^1(\Om_{\PML})^4} 
		&\leq C\sup_{\Psi \in H_0^1(\Om_{\PML})^4\ba\{0\}}\frac{\lvert \mathcal{B}_{\Om_{\PML}}
		((\wid{U}^{\top},\wid{P})^{\top},\Psi) \rvert}{\Vert \Psi \Vert_{H_0^1(\Om_{\PML})^4}}\\
		 &=C\sup_{\Psi \in H_0^1(\Om_{\PML})^4\ba\{0\}}\frac{\lvert \mathcal{B}_{\Om_{\PML}}
		 (\chi F,\Psi) \rvert}{\Vert \Psi \Vert_{H_0^1(\Om_{\PML})^4}}\\
		 &\leq C\sup\limits_{\Psi \in H_0^1(\Om_{\PML})^4\ba\{0\}}\frac{\Vert \chi F \Vert_{H^1(\Om_{\PML})^4}
		 \Vert \Psi \Vert_{H_0^1(\Om_{\PML})^4}}{\Vert \Psi \Vert_{H_0^1(\Om_{\PML})^4}}\\
		 &\leq C(d)\Vert F \Vert_{H^{1/2}(\pa B_2)^4},
	\end{aligned}	 	
\end{equation*}
which implies 		
\begin{equation*}
	\Vert (\wid{\bs{u}}^{\top},\wid{p})^{\top}-\chi F \Vert_{H^1(\Om_{\PML})^4}
	\leq C(d)\Vert F \Vert_{H^{1/2}(\pa B_2)^4}.
\end{equation*}
The desired estimate \eqref{stability_pmllayer} then follows from the triangle inequality.

\end{proof}
\end{theorem}

\section{Convergence analysis of the PML method}
\setcounter{equation}{0}

In this section, we make the following assumption on the thickness $d$ of the PML layer:
\begin{equation}\label{assum_d}
	\ov{l_j}-l_j\leq \frac{d}{2},j=1,2,3,\;{\rm and}\;d\;\text{has a positive lower bound}.
\end{equation}

To study the convergence of the uniaxial PML method, we introduce the DtN operator 
$\hat{\mathcal{N}}:H^{1/2}(\pa B_1)^4\to H^{-1/2}(\pa B_1)^4$ associated with the 
truncated PML problem \eqref{tPML1}-\eqref{tPML4}.	Given $G\in H^{1/2}(\pa B_1)^4$, define
\begin{equation}\label{dtnmap_pml}
	\hat{\mathcal{N}}G=\wid{\mathcal{R}}(\wid{\bs{u}}^{\top},\wid{p})^{\top},
\end{equation}
where $(\wid{\bs{u}}^{\top},\wid{p})^{\top}$ satisfies the equations in the PML layer
\begin{subequations}
	\begin{align}\label{dtn_pml1}
		\na \cdot( \wid{\sig}(\wid{\bs{u}})A)+\rho \om^2 J \wid{\bs{u}}-\g A \na\wid{ p}=0 
		\quad &{\rm in} \;\Om_{\PML} , \\ \label{dtn_pml2}
		\na \cdot (K\na\wid{ p})+qJ \wid{p}+{\bf{i}}\om \eta \na \cdot(A\wid{ \bs{u}})=0 
		\quad &{\rm in} \;\Om_{\PML},\\\label{dtn_pml3}
		(\wid{\bs{u}}^{\top},\wid{p})^{\top} =0 \quad &{\rm on} \; \pa B_2,\\ \label{dtn_pml4}
		(\wid{\bs{u}}^{\top},\wid{p})^{\top}=G \quad &{\rm on} \; \pa B_1.
	\end{align}
\end{subequations}
Here, $\wid{\mathcal{R}}$ is defined by
\begin{equation}
	\wid{\mathcal{R}}(\wid{\bs{u}}^{\top},\wid{p})^{\top}=
	\begin{bmatrix}
		\wid{T} & -\g \wid{\nu } \\
		0 &  \pa_{\wid{\nu}}
	\end{bmatrix}(\wid{\bs{u}}^{\top},\wid{p})^{\top}\quad {\rm on}\;\pa B_1
\end{equation}
where $\wid{\nu} =A\nu$, $\pa_{\wid{\nu}}\bs{u}=\nu \cdot (A\na \bs{u})$, and 
$\wid{T} \bs{u}=\wid{\sig}(\wid{\bs{u}})A \nu$. Observing that $A$ and $B$ reduce to 
the identity matrices on $\pa B_1$, thereby $\wid{\sig}(\wid{\bs{u}})=\sig(\wid{\bs{u}})$, 
hence we obtain $\wid{\mathcal{R}}(\wid{\bs{u}}^{\top},\wid{p})^{\top}
=\mathcal{R}(\wid{\bs{u}}^{\top},\wid{p})^{\top}$.

For $x\in\pa B_2,y\in\pa B_1$, it follows from \eqref{pmld} and the assumption \eqref{assum_d} that
\begin{equation}\label{pmld1}
	\Ip d(\wid{x},\wid{y})\geq \frac{\ds\sum_{j=1}^3\left(\lvert x_j-y_j\rvert 
	\left\lvert\int_{y_j}^{x_j}\al_j(t)dt\right\rvert+\zeta\left\lvert \int_{y_j}^{x_j}\al_j(t)dt
	\right\rvert^2\right)}{(1+\zeta\al_0)\sqrt{(x_1-y_1)^2+(x_2-y_2)^2+(x_3-y_3)^2}}\geq r_0 \al_0 d,
\end{equation}
where $r_0=\ds\frac{3}{4}\frac{d}{\sqrt{\sum_{j=1}^3(2l_j+d_j)^2}}$.

For $v\in W^{1,\infty }(\pa B_2) \cap H^{1/2}(\pa B_2)$, from the $H^{1/2}$-norm on $\pa B_i\;(i=1,2)$
\begin{equation*}
	\|v\|_{H^{1/2}(\pa B_i)} = \left( \|v\|_{L^2(\pa B_i)}^2 + \int_{\pa B_i} \int_{\pa B_i} 
	\frac{|v(x) - v(x')|^2}{|x - x'|^3} \, ds(x) \,ds(x') \right)^{\frac{1}{2}},
\end{equation*}
we have
\begin{equation*}
	\Vert v\Vert_{ H^{1/2}(\pa B_i)}\leq C\Big(\vert \pa B_i\vert^{1/2}\Vert v\Vert_{L^{\infty}(\pa B_i)}
	+\vert \pa B_i\vert^{1/2}{\rm diam}(B_i)^{1/2}\Vert\na v\Vert_{L^{\infty}(\pa B_i)}\Big),
\end{equation*}
where ${\rm diam}(B_i)$ denotes the diameter of domain $B_i$. If $d$ has a positive lower bound, we have
\begin{equation}\label{H1/2norm}
	\Vert v\Vert_{ H^{1/2}(\pa B_i)}\leq C\Big(d\Vert v\Vert_{L^{\infty}(\pa B_i)}
	+d^{3/2}\Vert\na v\Vert_{L^{\infty}(\pa B_i)}\Big).
\end{equation}

The following lemma on the decay property of the PML extension will play an important role in the 
convergence analysis of the PML method.
\begin{lemma}\label{thermoelasticPMLpotential}
For $\bs{f}=(f_1,f_2,f_3,f_4)^{\top}\in H^{1/2}(\pa B_1)^4$, under the assumption \eqref{assum_d} 
we have the estimate for the PML extension
\begin{equation}\label{E_pml_estimate}
	\Vert \wid{E}(\bs{f})\Vert_{{ H^{1/2}(\pa B_2)}^4}\leq C\Big[(1+\zeta \al_0)^2+\al_0^2\Big]^{\frac{m}{2}}
	P(d)e^{-r_0\al_0d}\Vert \bs{f}\Vert_{{H^{1/2}(\pa B_1)^4}},
\end{equation}
where $m>0$ is an integer, $r_0$ appears in \eqref{pmld1} and $P(d)$ is a function that grows at most polynomially.
\end{lemma}
\begin{proof}
Let $M_{(k)}$ and $M^{(k)}$ be the $k$-th row vector and column vector, respectively, of matrix $M$ over $\mat{C}$. 
For $x\in\pa B_2$, the $k$-th component of the PML extension $\wid{E}(\bs{f})$ 
(see \eqref{pmlextension} for its definition) is given by
\begin{equation}\label{kcomponent}
	\wid{E}(\bs{f})_k(x)=\frac{1}{2}\Big[\wid{\Psi}_{SL}(\mathcal{N}\bs{f})_k(x)-\wid{\Psi}_{DL}(\bs{f})_k(x)\Big],
\end{equation}
where
\begin{equation*}
	\begin{aligned}
		\wid{\Psi}_{SL}(\mathcal{N}\bs{f})_k(x)&=\int_{\pa B_1} \wid{\Phi}(x-y,\om)_{(k)}(\mathcal{N}\bs{f})(y)ds(y),\\
		\wid{\Psi}_{DL}(\bs{f})_k(x)&=\int_{\pa B_1}[\mathcal{S}(\pa_y,\nu(y),\om)
		\wid{\Phi}^{\top}(x-y,\om)]^{(k)}\bs{f}(y)ds(y).
	\end{aligned}
\end{equation*}
By using the estimate for $H^{1/2}$-norm \eqref{H1/2norm} on $\pa B_2$, we have
\begin{equation}\label{SL1}
	\begin{aligned}
		\Vert \wid{\Psi}_{SL}(\mathcal{N}\bs{f})_k\Vert_{ H^{1/2}(\pa B_2)}
		\leq\;&Cd\Big\Vert\int_{\pa B_1}\wid{\Phi}(x-y,\om)_{(k)}(\mathcal{N}{\bs{f}})(y)ds(y)
		\Big\Vert_{L^{\infty}(\pa B_2)}\\
		&+Cd^{3/2}\Big\Vert\na_x \int_{\pa B_1}\wid{\Phi}(x-y,\om)_{(k)}(\mathcal{N}{\bs{f}})(y)ds(y)
		\Big\Vert_{L^{\infty}(\pa B_2)},
	\end{aligned}
\end{equation}
and 
\begin{align}\label{DL1}
	\Vert\wid{\Psi}_{DL}(\bs{f})_k\Vert_{ H^{1/2}(\pa B_2)}
	\leq\; & Cd\Big\Vert \int_{\pa B_1}\Big[\mathcal{S}(\pa_y,\nu(y),\om)\wid{\Phi}^{\top}(x-y,\om)\Big]^{(k)}
	\bs{f}(y)ds(y)\Big\Vert_{L^{\infty}(\pa B_2)}\\\no
	&+Cd^{3/2}\Big\Vert\na_x\int_{\pa B_1}\Big[\mathcal{S}(\pa_y,\nu(y),\om)\wid{\Phi}^{\top}(x-y,\om)\Big]^{(k)}
	\bs{f}(y)ds(y)\Big\Vert_{L^{\infty}(\pa B_2)}.
\end{align}
By using the dual pair $\langle\cdot,\cdot\rangle_{H^{-1/2}(\pa B_1)\times H^{1/2}(\pa B_1)}$ and 
taking the maximum on $\pa B_2$ for \eqref{SL1} yield
\begin{align}\label{SL3}
	\Vert \wid{\Psi}_{SL}(\mathcal{N}\bs{f})_k\Vert_{ H^{1/2}(\pa B_2)} 
	\leq\;&Cd\max\limits_{x\in\pa B_2}\sum_{j=1}^4\Vert\wid{\Phi}(x-\cdot,\om)_{kj}\Vert_{{H^{1/2}(\pa B_1)}}
	\Vert f_j\Vert_{{H^{1/2}(\pa B_1)}}\\\no
	&+Cd^{3/2}\max\limits_{x\in\pa B_2}\max\limits_{t\in\{1,2,3\}}\sum_{j=1}^4\Big\Vert\frac{\pa }{\pa x_t}
	\wid{\Phi}(x-\cdot,\om)_{kj}\Big\Vert_{{H^{1/2}(\pa B_1)}}\Vert f_j\Vert_{{H^{1/2}(\pa B_1)}},
\end{align}
where the boundness of DtN operator $\mathcal{N}$ is used. The estimate for \eqref{DL1} can be similarly
obtained in combination with the fact that $H^{-1/2}(\pa B_1)$-norm is controlled by $H^{1/2}(\pa B_1)$-norm
for smooth functions:
\begin{align}\label{DL3}
	\Vert\wid{\Psi}_{DL}(\bs{f})_k\Vert_{ H^{1/2}(\pa B_2)}
	\leq\; & Cd\max\limits_{x\in\pa B_1}\sum_{i=1}^4\sum_{a=1}^4\Vert \mathcal{S}_{ia}
	\wid{\Phi}(x-\cdot,\om)_{ka}\Vert_{{H^{1/2}(\pa B_1)}}\Vert f_i\Vert_{{H^{1/2}(\pa B_1)}}\\\no
	&+Cd^{3/2}\max\limits_{x\in\pa B_1}\max\limits_{t\in\{1,2,3\}}\sum_{i=1}^4\sum_{a=1}^4
	\Big\Vert \frac{\pa}{\pa x_t}\mathcal{S}_{ia}\wid{\Phi}(x-\cdot,\om)_{ka}\Big\Vert_{{H^{1/2}(\pa B_1)}}
	\Vert f_i\Vert_{{H^{1/2}(\pa B_1)}}.
\end{align}
Using the $H^{1/2}$-norm estimate \eqref{H1/2norm} on $\pa B_1$, we obtain the following bounds 
for \eqref{SL3} and \eqref{DL3}
\begin{equation*}
	\begin{aligned}
		\Vert \wid{\Psi}_{SL}(\mathcal{N}\bs{f})_k\Vert_{ H^{1/2}(\pa B_2)} 
		\leq\; &Cd\Vert\bs{f}\Vert_{{H^{1/2}(\pa B_1)^4}}\sum_{j=1}^4\Big( d\Vert \wid{\Phi}(x-y,\om)_{kj}
		\Vert_{L^{\infty}(\pa B_1\times\pa B_2)}\\
		&+d^{3/2}\Vert\na_y \wid{\Phi}(x-y,\om)_{kj}\Vert_{L^{\infty}(\pa B_1\times\pa B_2)}\Big)\\
		&+Cd^{3/2}\Vert\bs{f}\Vert_{{H^{1/2}(\pa B_1)^4}}\max\limits_{t\in\{1,2,3\}}\sum_{j=1}^4 
		\Big(d\Big\Vert \frac{\pa }{\pa x_t}\wid{\Phi}(x-y,\om)_{kj}\Big\Vert_{L^{\infty}(\pa B_1\times\pa B_2)}\\
		&+d^{3/2}\Big\Vert\na_y \frac{\pa }{\pa x_t}\wid{\Phi}(x-y,\om)_{kj}\Big\Vert_{L^{\infty}(\pa B_1\times\pa B_2)}\Big),
	\end{aligned}
\end{equation*}
and
\begin{equation*}
	\begin{aligned}
		\Vert\wid{\Psi}_{DL}(\bs{f})_k\Vert_{ H^{1/2}(\pa B_2)} 
		\leq\; &Cd\Vert\bs{f}\Vert_{{H^{1/2}(\pa B_1)^4}}\sum_{i=1}^4\sum_{a=1}^3 
		\Big(d\Vert\mathcal{S}_{ia}\wid{\Phi}(x-y,\om)_{ka} \Vert_{L^{\infty}(\pa B_1\times\pa B_2)}\\
		&+d^{3/2}\Vert\na_y \mathcal{S}_{ia}\wid{\Phi}(x-y,\om)_{ka}\Vert_{L^{\infty}(\pa B_1\times\pa B_2)}\Big)\\
		&+Cd^{3/2}\Vert\bs{f}\Vert_{{H^{1/2}(\pa B_1)^4}}\max\limits_{t\in\{1,2,3\}}
		\sum_{i=1}^4\sum_{a=1}^3\Big( d\Big\Vert \frac{\pa}{\pa x_t}\mathcal{S}_{ia}
		\wid{\Phi}(x-y,\om)_{ka}\Big\Vert_{L^{\infty}(\pa B_1\times\pa B_2)}\\
		&+d^{3/2}\Big\Vert\na_y \frac{\pa}{\pa x_t}\mathcal{S}_{ia}\wid{\Phi}(x-y,\om)_{ka}
		\Big\Vert_{L^{\infty}(\pa B_1\times\pa B_2)}\Big).
	\end{aligned}
\end{equation*}
This, combined with the representation \eqref{kcomponent} and Lemma \ref{pmlfse} obtains the desired estimate \eqref{pmlfse}. 
The proof is thus complete.
	
\end{proof}

Based on Lemma \ref{thermoelasticPMLpotential}, we have the error estimate between DtN operators
$\mathcal{N}$ and $\hat{\mathcal{N}}$.
\begin{lemma}\label{thermoleasticDtNerror}
For $\bs{f}\in H^{1/2}(\pa B_1)^4$, we have
\begin{equation}\label{dtn_error}
	\Vert (\mathcal{N}-\hat{\mathcal{N}})\bs{f}\Vert_{{ H^{-1/2}(\pa B_1)}^4}
	\leq C\Big[(1+\zeta \al_0)^2+\al_0^2\Big]^{\frac{m}{2}}P(d)e^{-r_0\al_0d}\Vert\bs{f}\Vert_{{H^{1/2}(\pa B_1)^4}}, 
\end{equation}
where $m>0$ is an integer, $r_0$ appears in \eqref{pmld1} and $P(d)$ is a function that grows at most polynomially.
\end{lemma}

\begin{proof}
Given $\bs{f}\in H^{1/2}(\pa B_1)^4$, let $\wid{E}(\bs{f})$ be the PML extension potential defined in \eqref{pmlextension}.
From the definition of DtN operator $\mathcal{N}$ in \eqref{dtn}, we know 
\begin{equation*}
	\mathcal{N}\bs{f}=\mathcal{R} E(\bs{f})|_{\pa B_{1}}=\mathcal{R} \wid{E}(\bs{f})|_{\pa B_{1}}.
\end{equation*}
 Now let $\widetilde{U}$ be the solution of system \eqref{dtn_pml1}-\eqref{dtn_pml4} with $G=\bs{f}$ 
 on $\pa B_1$. From the definition of DtN operator $\hat{\mathcal{N}}$ in \eqref{dtnmap_pml}, we know 
\begin{equation*}
	\hat{\mathcal{N}}\bs{f}=\wid{\mathcal{R}}\wid{U}|_{\pa B_1}=\mathcal{R}\wid{U}|_{\pa B_1}.
\end{equation*} 
Then $(\mathcal{N}-\hat{\mathcal{N}})\bs{f}=\mathcal{R}\wid{W}|_{\pa B_1}$, where 
$\wid{W}=\wid{U}-\wid{E}(\bs{f})$ satisfies the system \eqref{PMLlayerstabilityequa1}-\eqref{PMLlayerconditioninner} 
with  $F=-\wid{E}(\bs{f})$ on $\pa B_2$. By Theorem \ref{pmllayer}, the following stability estimate 
\begin{equation}\label{W_stability}
	\Vert \wid{W} \Vert_{H^1(\Om_{\PML})^4}\leq C(d) \Vert \wid{E}(\bs{f})\Vert_{{ H^{1/2}(\pa B_2)}^4}
\end{equation}
therefore holds, where the constant $C(d)$ depends on the thickness $d$ and grows at most polynomially in $d$.

Let $\Psi \in H^{1/2}(\pa B_{1})^{4}$ be extended to a function in $H^{1}(\Om_{\text{PML}})^{4}$ 
(still denoted by $\Psi$) such that $\Psi = 0$ on $\pa B_{2}$ and $\| \Psi \|_{H^{1/2}(\pa B_{1})^{4}} 
\ge C\| \Psi \|_{H^{1}(\Om_{\text{PML}})^{4}}$. Then it follows from the integration by parts and \eqref{W_stability} that
\begin{equation*}
	\begin{aligned}
		\lvert\langle\mathcal{R}\wid{W}|_{\pa B_1},\Psi\rangle_{\pa B_1}\rvert
		&=\lvert \langle(\hat{\mathcal{N}}-\mathcal{N})\bs{f},\Psi\rangle_{\pa B_1}\rvert
		= \lvert \mathcal{B}_{\PML}(\wid{W},\Psi)\rvert\\
		&\leq C\Vert \wid{W} \Vert_{H^1(\Om_{\PML})^4}\Vert \Psi \Vert_{H^1(\Om_{\PML})^4}\\
		&\leq C(d)\Vert \wid{E}(\bs{f})\Vert_{{ H^{1/2}(\pa B_2)}^4}\Vert \Psi\Vert_{{ H^{1/2}(\pa B_1)}^4},
	\end{aligned}
\end{equation*}
where the sesquilinear form $\mathcal{B}_{\PML}(\cdot,\cdot)$ is defined in \eqref{biform_B} 
for $D=\Om_{\PML}$. The it follows from the definition of $H^{-1/2}(\pa B_1)^4$ that 
\begin{equation*}
	\Vert\mathcal{R}\wid{W}|_{\pa B_1}\Vert_{{ H^{-1/2}(\pa B_1)}^4} 
	= \sup\limits_{\Psi \in H^{1/2}(\pa B_{1})^{4}\ba\{0\}}\frac{\lvert
	\langle\mathcal{R}\wid{W}|_{\pa B_1},\Psi\rangle_{\pa B_1}\rvert}
	{\Vert \Psi\Vert_{{ H^{1/2}(\pa B_1)}^4}}\leq C(d)\Vert \wid{E}(\bs{f})\Vert_{{ H^{1/2}(\pa B_2)}^4}.
\end{equation*}
This, combined with Lemma \ref{thermoelasticPMLpotential} arrive at the desired estimate \eqref{dtn_error}.

\end{proof}

Now we can establish the main result on the exponential convergence of the PML method.
\begin{theorem}[Convergence Theorem]\label{thm:convergence}
Let conditions \eqref{pml_pc} be satisfied and $\al_0d$ be sufficiently large. For given 
${\bf Q} \in H^1(\R^3\ba\ov{\Om})^3$ with compact support inside $B_1$, let $U$ and $\wid{U}$ 
be the solutions of the problems \eqref{matrix_eq} with Kupradze radiation conditions \eqref{Kupradze} 
and \eqref{tPML1}-\eqref{tPML4}, respectively. Then we have the following error estimate
\begin{equation}\label{solution_err}
	\Vert \wid{U}-U \Vert_{H^1(\Om_1)^4}\leq C\Big[(1+\zeta \al_0)^2+\al_0^2\Big]^{\frac{m}{2}}P(d)
	e^{-r_0\al_0d}\Vert \wid{U}\Vert_{{H^{1/2}(\pa B_1)^4}},
\end{equation}
where $m>0$ is an integer, $r_0$ appears in \eqref{pmld1} and $P(d)$ is a function that grows at most polynomially.
\begin{proof}
Noting that ${\bf Q} \in H^1(\R^3\ba\ov{\Om})^3$ is compactly supported inside $B_1$, then $U$ and 
$\wid{U}$ satisfy the following variational problems
\begin{align}\label{scattering}
	\mathcal{B}_1(U,\Psi)&=-\int_{\Om_1}{\bf Q}\cdot \ov{\Psi}dx,
	\quad \forall\;\Psi \in H^{1}_{\pa B_1}(\Om_1)^4 \\ \label{pmlproblem}
	\mathcal{B}_{\Om_2}(\wid{U},\Psi)&=-\int_{\Om_1}{\bf Q}\cdot \ov{\Psi}dx,
	\quad \forall\;\Psi \in H_0^1(\Om_2)^4.
\end{align}
Simple decomposition of $\mathcal{B}_{\Om_2}$ and integration by parts imply
\begin{equation}\label{B_decom}
	\begin{aligned}
		\mathcal{B}_{\Om_2}(\wid{U},\Psi)=\;&	\mathcal{B}_1(\wid{U},\Psi)
		+\langle\mathcal{N}\wid{U},\Psi\rangle_{\pa B_1}+\mathcal{B}_{\Om_{\PML}}(\wid{U},\Psi)\\
		=\;&\mathcal{B}_1(\wid{U},\Psi)+\langle\mathcal{N}\wid{U},\Psi\rangle_{\pa B_1}
		+\langle-\hat{\mathcal{N}}\wid{U},\Psi\rangle_{\pa B_1}.
	\end{aligned}		
\end{equation}
Subtracting the equations \eqref{scattering} and \eqref{pmlproblem} and using \eqref{B_decom} gives
\begin{equation*}
	\begin{aligned}
		\mathcal{B}_1(U-\wid{U},\Psi)=\langle(\mathcal{N}-\hat{\mathcal{N}})\wid{U},\Psi\rangle_{\pa B_1}.
	\end{aligned}	
\end{equation*}
This, combined with the trace theorem gives
\begin{equation*}
	\begin{aligned}
		\lvert \mathcal{B}_1(U-\wid{U},\Psi)\rvert &\leq C\Vert (\mathcal{N}-\hat{\mathcal{N}})\wid{U}
		\Vert_{ H^{-1/2}(\pa B_1)^4}\Vert\Psi \Vert_{H^{1}(\Om_1)^4},
	\end{aligned}		
\end{equation*}
which obviously implies
\begin{equation}\label{eq:label}
	\frac{\lvert \mathcal{B}_1(U-\wid{U},\Psi)\rvert}{\Vert\Psi \Vert_{H^{1}(\Om_1)^4}}
	\leq C\Vert (\mathcal{N}-\hat{\mathcal{N}})\wid{U}\Vert_{ H^{-1/2}(\pa B_1)^4}.
\end{equation}
The desired error estimate \eqref{solution_err} follows from Lemma \ref{thermoleasticDtNerror} 
and the inf-sup condition \eqref{infsup_B1} for the sesquilinear form $ \mathcal{B}_1$.
	
\end{proof}

\begin{remark}[Uniform boundedness of $\wid{U}$ with respect to large PML parameters]
{\rm Since $\wid{U}$ is the solution of the truncated PML problem \eqref{tPML1}-\eqref{tPML4},
$\|\wid{U}\|_{H^{1/2}(\pa B_1)^4}$ naturally depends on the PML parameters $d$ and $\alpha_0$. 
A natural question is: how does $\|\wid{U}\|_{H^{1/2}(\pa B_1)^4}$ grow as $d$ and $\alpha_0$ vary? 
We discuss this problem below under the conditions of Theorem \ref{thm:convergence}.

From the triangle inequality
\begin{equation*}
\|\wid{U}\|_{H^1(\Om_1)^4} - \|U\|_{H^1(\Om_1)^4} \leq \|\wid{U} - U\|_{H^1(\Om_1)^4},
\end{equation*}  
and trace theorem
\begin{equation*}
\frac{1}{C_{\Om_1}} \|\wid{U}\|_{H^{1/2}(\pa B_1)^4} \leq \|\wid{U}\|_{H^1(\Om_1)^4},
\end{equation*}  
where $C_{\Om_1}$ is a positive constant independent of $d$ and $\al_0$. 
It then follows from the error estimate \eqref{solution_err} that  
\begin{equation*}
\frac{1}{C_{\Om_1}} \|\wid{U}\|_{H^{1/2}(\pa B_1)^4} - C\left[(1 + \zeta \al_0)^2 
+ \al_0^2\right]^{\frac{m}{2}} P(d)e^{-r_0\al_0 d} \|\wid{U}\|_{H^{1/2}(\pa B_1)^4} \leq \|U\|_{H^1(\Om_1)^4},
\end{equation*}  
where $m$ is a positive integer. It can be observed that when $\al_0 d$ is sufficiently large, 
$\|\wid{U}\|_{H^{1/2}(\pa B_1)^4}$ remains uniformly bounded.}
\end{remark}

\end{theorem}

\section{Conclusion}
\setcounter{equation}{0}
In this paper, we studied the uniaxial PML method in the Cartesian coordinates for 3D time-harmonic 
thermoelastic scattering problems. Under certain constraints on model parameters, we showed the existence 
of unique solution except possibly for a discrete set of frequencies, for the PML problems both in the 
truncated domain and PML layer by using the classical analytic Fredholm theory. Furthermore, the exponential 
convergence of the uniaxial PML method was established in terms of the thickness and absorbing parameters of 
the PML layer. The proof was based on the error estimate between the Dirichlet-to-Neumann (DtN) operators for
the original scattering problem and the truncated PML problem, which depends on the exponential decay properties
for the PML extensions and modified fundamental solutions.

Our method can be extended to other scattering problems such as the scattering by waveguide structures
or poroelastic obstacles. It is also interesting to study the convergence of proposed uniaxial PML method 
with distinct boundary conditions at the outer PML interface, and spherical PML method for time-harmonic 
thermoelastic scattering problems. We hope to report such results in the near future.

\section*{Acknowledgements}

This work is partly supported by the NNSF of China grants 12431016 and 12201033.


\begin{thebibliography}{99}

\bibitem{ASZ2024} E.S. Athanasiadou, V. Sevroglou and S. Zoi, Thermoelastic wave scattering by a multi-layered object, 
{\em Comput. Math. Appl. \bf 163} (2024), 186-200.

\bibitem{BW2005} G. Bao and H. Wu, Convergence analysis of the perfectly matched layer problems for time-harmonic 
Maxwell's equations, {\em SIAM J. Numer. Anal. \bf 43} (2005), 2121-2143.

\bibitem{BXY2019} G. Bao, L. Xu and T. Yin, Boundary integral equation methods for the elastic and thermoelastic waves 
in three dimensions, {\em Comput. Methods Appl. Mech. Engrg. \bf 354} (2019), 464-486.

\bibitem{Berenger1994} J.P. B\'erenger, A perfectly matched layer for the absorption of electromagnetic waves, 
{\em J. Comput. Phys. \bf 114} (1994), 185-200.

\bibitem{BP2012} J.H. Bramble and J.E. Pasciak, Analysis of a Cartesian PML approximation to the three dimensional 
electromagnetic wave scattering problem, {\em Int. J. Numer. Anal. Model. \bf 9} (2012), 543-561.

\bibitem{BP2013} J. H. Bramble and J.E. Pasciak, Analysis of a Cartesian PML approximation to acoustic scattering 
problems in $\R^2$ and $\R^3$, {\em J. Comput. Appl. Math. \bf 247} (2013), 209-230.

\bibitem{Bramble2007} J.H. Bramble and J.E. Pasciak, Analysis of a finite PML approximation for the three
dimensional time-harmonic Maxwell and acoustic scattering problems, {\em Math. Comp. \bf76} (2007), 597-614.

\bibitem{BPT2010} J.H. Bramble, J.E. Pasciak and D. Trenev, Analysis of a finite PML approximation to the three 
dimensional elastic wave scattering problem, {\em Math. Comp. \bf 79} (2010), 2079-2101.

\bibitem{Cakoni2000} F. Cakoni, Boundary integral method for thermoelastic screen scattering problem in $\R^3$, 
{\em Math. Meth. Appl. Sci. \bf 23} (2000), 441-466.

\bibitem{Cakoni1999} F. Cakoni and G. Dassios, The Atkinson-Wilcox theorem in thermoelasticity, 
{\em Q. Appl. Math. \bf 57} (1999), 771-795.

\bibitem{Chen2009} Z. Chen, Convergence of the time-domain perfectly matched layer method for acoustic 
scattering problems, {\em Int. J. Numer. Anal. Model. \bf 6} (2009), 124-146.

\bibitem{CCZ2013} Z. Chen, T. Cui and L. Zhang, An adaptive anisotropic perfectly matched layer method for 
3-D time harmonic electromagnetic scattering problems, {\em Numer. Math. \bf 125} (2013), 639-677.

\bibitem{CL2005} Z. Chen and X. Liu, An adaptive perfectly matched layer technique for time-harmonic 
scattering problems, {\em SIAM J. Numer. Anal. \bf 43} (2005), 645-671.

\bibitem{CW2008} Z. Chen and X. Wu, An adaptive uniaxial perfectly matched layer method for time-harmonic 
scattering problems, {\em Numer. Math. Theor. Meth. Appl. \bf 1} (2008), 113-137.

\bibitem{CW2012} Z. Chen and X. Wu, Long-time stability and convergence of the uniaxial perfectly matched 
layer method for time-domain acoustic scattering problems, 
{\em SIAM J. Numer. Anal. \bf 50} (2012), 2632-2655.

\bibitem{CXZ2016} Z. Chen, X. Xiang and X. Zhang, Convergence of the PML method for elastic wave scattering 
problems, {\em Math. Comp. \bf 85} (2016), 2687-2714.

\bibitem{CZ2010} Z. Chen and W. Zheng, Convergence of the uniaxial perfectly matched layer method for 
time-harmonic scattering problems in two-layered media, 
{\em SIAM J. Numer. Anal. \bf 48} (2010), 2158-2185.

\bibitem{CZ2017} Z. Chen and W. Zheng, PML method for electromagnetic scattering problem in a two-layer 
medium, {\em SIAM J. Numer. Anal. \bf 55} (2017), 2050-2084.

\bibitem{Chew1994} W.C. Chew and W.H. Weedon, A 3D perfectly matched medium from modified Maxwell's equations 
with stretched coordinates, {\em Microw. Opt. Technol. Lett. \bf 7} (1994), 599-604.

\bibitem{Fathi2015} A. Fathi, L.F. Kallivokas and B. Poursartip, Full-waveform inversion in three-dimensional 
PML-truncated elastic media, {\em Comput. Meth. Appl. Mech. Eng. \bf 296} (2015), 39-72.

\bibitem{HSZ2003} T. Hohage, F. Schmidt and L. Zschiedrich, Solving time-harmonic scattering problems based 
on the pole condition II: convergence of the PML method, {\em SIAM J. Math. Anal. \bf 35} (2003), 547-560.

\bibitem{Hou2022} W. Hou, L.Y. Fu and J.M. Carcione, Reflection and transmission of thermoelastic waves in 
multilayered media, {\em Geophysics \bf 87} (2022), MR117-MR128.

\bibitem{Hou2021} W. Hou, L.Y Fu, J.M. Carcione and Z. Wang, Simulation of thermoelastic waves based on the 
Lord-Shulman theory, {\em Geophysics \bf 86} (2021), T155-T164.

\bibitem{Jiang2017} X. Jiang, P. Li, J. Lv and W. Zheng, An adaptive finite element PML method for the elastic 
wave scattering problem in periodic structures, {\em ESAIM: Math. Model. Numer. Anal. \bf 51} (2017), 2017-2047.

\bibitem{Jiang2018} X. Jiang, P. Li, J. Lv and W. Zheng, Convergence of the PML solution for elastic wave 
scattering by biperiodic structures, {\em Commun. Math. Sci. \bf 16} (2018), 987-1016.

\bibitem{Kim2010} S. Kim and J.E. Pasciak, Analysis of a Cartesian PML approximation to acoustic scattering 
problems in $\R^2$, {\em J. Math. Anal. Appl. \bf 370} (2010), 168-186.

\bibitem{Kupradze1979} V.D. Kupradze, T.G. Gegelia, M.O. Basheleishvili and T.V. Burchuladze,
{\em Three-dimensional Problems of the Mathematical Theory of Elasticity and Thermoelasticity}, 
North-Holland, Amsterdam, 1979.

\bibitem{Lassas1998} M. Lassas and E. Somersalo, On the existence and convergence of the solution of PML 
equations, {\em Computing \bf 60} (1998), 229-241.

\bibitem{Lei2023} W. Lei, Y. Liu, G. Li, S. Zhu, G. Chen and C. Li, 2D frequency-domain finite-difference 
acoustic wave modeling using optimized perfectly matched layers, 
{\em Comput. Meth. Appl. Mech. Eng. \bf 88} (2023), F1-F13.

\bibitem{Pakravan2014} A. Pakravan, J. W. Kang, A Gauss-Newton full-waveform inversion for material profile 
reconstruction in 1D PML-truncated solid media, {\em KSCE J. Civ. Eng. \bf 18} (2014), 1792-1804.

\bibitem{Pakravan2017} A. Pakravan, J. W. Kang and C. M. Newtson, A Gauss-Newton full-waveform inversion in 
PML-truncated domains using scalar probing waves, {\em J. Comput. Phys. \bf 350} (2017), 824-846.

\bibitem{Savare1998} G. Savare, Regularity results for elliptic equations in Lipschitz domains, 
{\em J. Funct. Anal. \bf 152} (1998),  176-201.

\bibitem{Wang2022} E. Wang, J.M. Carcione and J. Ba, Wave simulation in partially saturated porothermoelastic 
media, {\em IEEE Trans. Geosci. Remote Sensing \bf 60} (2022), 1-14.

\bibitem{wang2025} Y. Wang, P. Li, L. Xu and T. Yin, An adaptive Dirichlet-to-Neumann finite element method 
for the thermoelastic scattering problem, {\em J. Comput. Phys. \bf 534} (2025), 114016.

\bibitem{wyz2020} C. Wei, J. Yang and B. Zhang, Convergence analysis of the PML method for time-domain 
electromagnetic scattering problems, {\em SIAM J. Numer. Anal. \bf 58} (2020), 1918-1940.

\bibitem{wyz2021} C. Wei, J. Yang and B. Zhang, Convergence of the uniaxial PML method for time-domain 
electromagnetic scattering problems, {\em ESAIM: Math. Model. Numer. Anal. \bf 55} (2021), 2421-2443.

\bibitem{Yang2024} S. Yang, G. Wu, J. Shan and H. Liu, Simulation of seismic waves in fluid-solid coupled 
thermoelastic media, {\em Geophysics \bf 89} (2024), T263-T274.

\bibitem{Zhu2024} T. Zhu, C. Wei, and J. Yang, The time-harmonic electromagnetic wave scattering by a 
biperiodic elastic body, {\em Math. Meth. Appl. Sci. \bf 47} (2024), 6354-6381.

\end{thebibliography}
\end{document}